\documentclass{amsart}
\usepackage{amsfonts}
\usepackage{amssymb}
\usepackage{esint}
\usepackage{color}
\usepackage{graphicx}
\usepackage[bookmarksnumbered,colorlinks, linkcolor=blue, citecolor=red, bookmarks, breaklinks]{hyperref}
\usepackage[left=3.5cm,right=3.5cm]{geometry}
\newtheorem{thm}{Theorem}[section]
\newtheorem{lemma}[thm]{Lemma}
\newtheorem{prop}[thm]{Proposition}

\theoremstyle{definition}

\newtheorem{definition}[thm]{Definition}

\numberwithin{equation}{section}

\def\dys{\displaystyle}
\numberwithin{equation}{section}
\newcommand{\R}{\mathbb{R}}

\newcommand{\lam}{\lambda}

\def \R {\mathbb{R}}

\def \L {\mathcal{L}}

\def \dist {\mathrm{dist}}
\def \div {\mathrm{div}}
\def \E {\mathcal{E}}

\def\an{\alpha_n}
\def\bn{\beta_n}
\def\pn{p_n}
\def\vn{v_n}
\def\un{u_n}

\def\elle#1{L^{#1}(\Omega)}
\def\w#1#2{W_0^{#1,#2}(\Omega)}

\def\elle#1{L^{#1}(\Omega)}

\def\io{\int_{\Omega}}
\def\norma#1#2{\|#1\|_{\lower 4pt \hbox{$\scriptstyle #2$}}}
\def\un{u_n}

\def\ui{u_{\infty}}
\def\vi{v_{\infty}}

\def\an{\alpha_n}

\def\pn{p_n}

\def\bn{\beta_n}

\def\finedim
\def\gw{G_{\tilde{k}}(w_n)}

\def\R{I \!\!R}

\def\elle#1{L^{#1}(\Omega)}

\newcommand{\defeq}{\mathrel{\mathop:}=}

\begin{document}
	
\title[Local/nonlocal $p-$eigenvalue problem and its asymptotic limit as $p\to \infty$]{A system of local/nonlocal $p-$Laplacians: the eigenvalue problem and its asymptotic limit as $p \to \infty$}

\author[S. Buccheri, J.V. da Silva and L.H de Miranda]{S. Buccheri, J.V. da Silva and L.H de Miranda}

\address{Faculty of Mathematics - University of Vienna
\hfill\break \indent Oskar-Morgenstern-Platz 1, 1090 Vienna, Austria.}
\email[S. Buccheri]{stefano.buccheri@univie.ac.at}

\address{Departamento de Matem\'atica - Instituto de Matem\'{a}tica, Estat\'{i}stica e Computa\c{c}\~{a}o Cient\'{i}fica \hfill\break \indent Universidade Estadual de Campinas - UNICAMP
\hfill\break \indent Cidade Universit\'{a}ria Zeferino Vaz, 13083-859, Campinas - SP - Brazil.}
\email[J.V. da Silva]{jdasilva@unicamp.br}

\address{Departamento de Matem\'atica - Instituto de Ci\^{e}ncias Exatas - Universidade de Bras\'{i}lia
\hfill\break \indent Campus Universit\'{a}rio Darcy Ribeiro, 70910-900, Bras\'{i}lia - DF - Brazil.}
\email[L.H. de Miranda]{demiranda@unb.br}



\begin{abstract} In this work, given $p\in (1,\infty)$, we prove the existence and simplicity of the first eigenvalue $\lambda_p$ and its corresponding eigenvector $(u_p,v_p)$, for the following local/nonlocal PDE system \begin{equation}\label{Eq0}
\left\{
\begin{array}{rclcl}
-\Delta_p u + (-\Delta)^r_p u & = & \frac{2\alpha}{\alpha+\beta}\lam |u|^{\alpha-2}|v|^{\beta}u & \mbox{in} & \Omega \\
-\Delta_p v + (-\Delta)^s_p v& = & \frac{2\beta}{\alpha+\beta}\lam |u|^{\alpha}|v|^{\beta-2}v & \mbox{in} & \Omega \\
  u& =& 0&\text{ on } & \mathbb{R}^N \setminus \Omega\\
  v& =& 0&\text{ on } & \mathbb{R}^N \setminus \Omega,
\end{array}
\right.
\end{equation}
where $\Omega\subset \R^N$ is a bounded open domain, $0<r, s<1$ and $\alpha(p)+\beta(p) = p$.
Moreover, we address the asymptotic limit as $p \to \infty$, proving the explicit geometric characterization of the corresponding first $\infty-$eigenvalue, namely $\lambda_{\infty}$, and the uniformly convergence of the pair $(u_p,v_p)$ to the $\infty-$eigenvector $(u_{\infty},v_{\infty})$. Finally, the triple $(u_{\infty},v_{\infty},\lambda_{\infty})$ verifies, in the viscosity sense, a limiting PDE system.
\end{abstract}	
\maketitle

\section{Introduction}

In this manuscript we study the following eigenvalue problem for a system of equations driven by the combination between quasilinear elliptic operators with $p-$structure, having simultaneous local and nonlocal diffusion
\begin{equation}\label{Eq1}
\left\{
\begin{array}{rclcl}
-\Delta_p u + (-\Delta)_p^r u & = & \frac{2\alpha}{\alpha+\beta}\lam |u|^{\alpha-2}|v|^{\beta}u & \mbox{in} & \Omega \\
-\Delta_p v + (-\Delta)_p^s v& = & \frac{2\beta}{\alpha+\beta}\lam |u|^{\alpha}|v|^{\beta-2}v & \mbox{in} & \Omega \\
  u& =& 0&\text{ on } & \mathbb{R}^N \setminus \Omega\\
  v& =& 0&\text{ on } & \mathbb{R}^N \setminus \Omega,
\end{array}
\right.
\end{equation}
where $\Omega\subset \mathbb{R}^N$ is a bounded, open connected domain, $p\in(1,\infty)$ and the parameters $r,s,\alpha$ and $\beta$ satisfy
\begin{equation}\label{parass}
   r,s\in(0,1) \ \ \mbox{and} \ \ \ \alpha(p),\beta(p)\ge1 \ \ \ \mbox{such that } \alpha(p) + \beta(p) = p.
\end{equation}

We recall that
\begin{equation}\label{defop}
  \Delta_p u=\mbox{div}(|\nabla u|^{p-2}\nabla u) \ \ \ \mbox{and} \ \ \ (-\Delta)_p^r u \defeq 2 \text{P.V.} \int_{\mathbb{R}^N} \frac{|u(x)-u(y)|^{p-2}(u(x)-u(y))}{|x-y|^{N+rp}}dy\footnote{The constant $2$ is related to the natural definition of weak solution for equations involving the fractional Laplacian.},
\end{equation}
where the integro-differential operator in \eqref{defop} is considered in the Cauchy principal value sense.

The main purpose of the second part of this manuscript concerns the study of the asymptotic behavior of any family of weak solutions $(u_p, v_p, \lambda_p)$ to \eqref{Eq1} as $p \to \infty$, as well as the geometric characterization of corresponding $\infty-$eigenvalue. Moreover, in contrast to purely local or nonlocal systems, the present case imposes some extra accurate analysis which we will clarify soon.

\subsection{An overview on the existing literature}

As it is very well-known, the local quasilinear operator in \eqref{defop}, namely the \textit{$p-$Laplacian}, arises from various phenomena in applied mathematics as reaction-diffusion and absorption processes, formation of dead-cores, non-newtonian flows and game theoretical methods in PDEs, just to mention a few (see D\'{i}az's monograph \cite{Diaz}, or   \cite{BdaSR19}, \cite{daSS18} and the references therein). Moreover, elliptic integro-differential operators, like the one in \eqref{defop}, have an intrinsic mathematical significance and a strong relation with a large variety of applications.  For instance they appear in  stochastic processes of L\`evy type, image processing and in a number of nonlocal diffusion and free boundary problems. The interested reader is referred to \cite{BV16}, \cite{caffa}, \cite{daSR19}, \cite{daSS19}, \cite{GiO}, \cite{ILPS}, \cite{MS16},  and references therein.

Furthermore, eigenvalue problems have been a classical topic of investigation and have received considerable attention along the past decades by several authors. Knowing that it is impossible to list a comprehensive literature on this theme, we just mention their strong relationship with bifurcation theory, resonance problems, spectral optimization problems and also with applied sciences, such as fluid and quantum mechanics. Without intention of being complete,  see \cite{KL}, \cite{Le}, \cite{Lindq90}, \cite{RSdaS18} and the references therein for further details.

In addition, since the seminal work \cite{BDM}, there has been an increasing interest for the limiting behaviour of problems related to the $p-$Laplacian operator as $p\to\infty$. Indeed, in \cite{BDM} the authors investigate the behaviour of the solutions to $-\Delta_p u_p= f$, with Dirichlet boundary conditions, as $p$ goes to infinity. Later on, in \cite{FIN99} and \cite{JLM99} the infinite-eigenvalue problem is addressed by taking the limit of
\begin{equation}\label{eq.p}
\left\{
\begin{array}{rclcl}
  -\Delta_p u_p & = & \lambda_p  |u_p|^{p-2} u_p & \text{in} & \Omega \\
  u_p & = & 0 & \text{on} & \partial \Omega,
\end{array}
\right.
\end{equation}
as $p\to \infty$. Remark that $\lambda_p$, the first eigenvalue of \eqref{eq.p}, is isolated, simple and characterized variationally by minimizing the following \textit{Rayleigh quotient}
 \begin{equation}\label{1er.p}
   \displaystyle \lambda_p \defeq \inf_{u \in W^{1,p}_0(\Omega) \setminus\{0\}}
   \frac{\| \nabla u\|_{L^p(\Omega)}^p}{\|u\|_{L^p(\Omega)}^p}>0.
\end{equation}
By letting $p\to \infty$ in \eqref{1er.p}, it is proved in \cite{JLM99} the following geometric characterization for the {\bf first $\infty-$eigenvalue}
\begin{equation} \label{lam.inf}
\lam_{\infty} \defeq \lim_{p\to\infty} \left(\lam_{ p}\right)^{\frac{1}{p}} =\inf_{v \in W^{1, \infty}_0(\Omega)\setminus\{0\}} \frac{\|\nabla v\|_{L^{\infty}(\Omega)}}{\|v\|_{L^{\infty}(\Omega)}}= \frac{1}{\mathrm{R}}.
\end{equation}
where $\displaystyle \mathrm{R} \defeq \max_{x \in \Omega} \dist(x, \partial \Omega)$, the radius of the largest ball contained in $\Omega$. Further, there exists a subsequence of $\{u_p\}_{p>1}$  that converges uniformly in $\Omega$ to $u_\infty$, viscosity solution of
\begin{equation}\label{eq.infty.p}
\left\{
\begin{array}{rclcl}
  \min\Big\{-\Delta_\infty u_\infty, |\nabla u_\infty|-\lam_{\infty}u_\infty\Big\} & = & 0 & \text{in} & \Omega \\
  v_\infty & = & 0 & \text{on} & \partial \Omega.
\end{array}
\right.
\end{equation}
In the literature $u_{\infty}$ called the \textbf{$\infty-$ground state} or \textbf{$\infty-$eigenfuction} associated to  $\lam_{\infty}$.
In general, \eqref{eq.infty.p} has multiplicity of solutions, what is intrinsically connected to the geometry of $\Omega$, see \cite{HSY}, \cite{JLM99},  \cite{JLM01} and \cite{Yu}. {Finally, we must also quote \cite{RS15} and \cite{RS16} concerning the asymptotic limit as $p \to \infty$ of the first eigenvalue for the $p-$Laplacian with Neumann and mixed boundary conditions}.

On the other hand, in \cite{FP14}, among other results, the authors obtain existence and simplicity of the first eigenvalue for a class of nonlocal operators whose model is
\begin{equation}\label{26-11}
\left\{
\begin{array}{rclcl}
  (-\Delta)^s_p u & = & \lambda  |u|^{p-2} u & \text{in} & \Omega \\
  u & = & 0 & \text{on} & \R^N \setminus \Omega,
\end{array}
\right.
\end{equation}
namely, the eigenvalue problem for the fractional $p-$Laplacian. In this fractional scenario, the authors of \cite[Proposition 20 and Section 8]{LL14} perform a complete study of the asymptotic limit as $p \to \infty$ of \eqref{26-11}, establishing in the nonlocal setting similar results to \eqref{lam.inf} and \eqref{eq.infty.p}.\\
Let us add that the same type of analysis has been carried out for systems driven by local or nonlocal $p-$Laplacians  in \cite{BRS} and \cite{DelPR19}, respectively. {Lastly, we should also quote \cite{DelPRSS15}, where the authors find an interpretation via optimal mass transport theory for limits of eigenvalue problems for the fractional $p-$Laplacian as $p \to \infty$}.

Finally, in \cite{DelPFR} the following local/nonlocal eigenvalue problem has been addressed
\begin{equation}
\nonumber
\left\{
\begin{array}{rclcl}
  \displaystyle -\Delta_p u - \int_{\mathbb{R}^N} \mathcal{J}(x-y)|u(x)-u(y)|^{p-2}(u(y)-u(x))dy & = & \lambda|u|^{p-2}u & \mbox{in} & \Omega \\
  u & = & 0 & \mbox{on} & \mathbb{R}^N\setminus \Omega,
\end{array}
\right.
\end{equation}
where $\mathcal{J}\colon \mathbb{R}^N \to \mathbb{R}_{+}$ is a nonsingular, radially symmetric, nonnegative and \emph{compactly supported} kernel. We point out that, also in this case, a geometric characterization of the limit eigenvalue is obtained, together with a limit equation having a more complex structure than in \eqref{eq.infty.p}.\\

Despite the latter reference, see also \cite{daSS19-2} and \cite{GS19}, problems with simultaneous local and nonlocal characters have been far less studied in the literature. In this framework, the main contributions of the present paper consist in the investigation of the existence and simplicity for the first variational eigenvalue of system \eqref{Eq1}, as well as its asymptotic behaviour, if $p\to \infty$. Roughly speaking, with respect to the previous literature, the main differences in dealing with \eqref{Eq1} arise from the interaction between the local and nonlocal operator and the fact that the kernel of the fractional $p-$Laplacian is singular and not compactly supported.

Indeed, before the analysis of the limit case, we prove that for any fixed $p \in (1,\infty)$ there exists a triple $(u_p,v_p,\lambda_p)$, solution to system \eqref{Eq1}, where $u_p,v_p>0$ in $\Omega$, and  $\lambda_p$ is the corresponding variational first eigenvalue. We also prove that $\lambda_p$ is simple, which means that,
if $(u,v)$ and $(\tilde u, \tilde v)$ are two pairs of solutions of \eqref{Eq1} with $\lambda=\lambda_p$, then there exists $k\in\mathbb{R}$ such that $(\tilde u, \tilde v)=k(u,\pm v)$. We emphasize that, in general, obtaining the simplicity for the first eigenvalue for nonlinear systems is not a trivial task. For this reason, we had to prove a version of the Maximum Principle and to develop other auxiliary results which allow us to overcome such issues, see Sections \ref{opr} \& \ref{Maximum Principle} and the end of the proof of Theorem \ref{MThm1}.

After that, by using the variational characterization of $\lambda_p$, we prove that
\[
  \left(\lambda_p\right)^{\frac{1}{p}}\to  \lambda_\infty := \max\left\{ \frac{1}{R} , \frac{1}{R^{\Gamma r+(1-\Gamma)s}}\right\}, \mbox{ as } p \to\infty,
\]
for  $\Gamma\in(0,1)$ such that
\[
\dfrac{\alpha(p)}{p}\to \Gamma \mbox{ and } \dfrac{\beta(p)}{p}\to 1-\Gamma, \mbox{ as } p \to\infty,
\]
where $R$ is the radius of the largest ball contained in $\Omega$, and $r$, $s$, $\alpha$ and $\beta$ are defined in \eqref{parass}. Observe that the previous limit exhibits the interplay between the geometry of the domain and the local and nonlocal operators, see Theorem \ref{MThm2} and the comments below. Finally, in Theorem \ref{MThm4} we show that $(u_p,v_p)$ converges uniformly to $(u_\infty,v_\infty)$ a viscosity solution to the  asymptotic equation associated to \eqref{Eq1}, see \eqref{EqLim} for the precise formulation.

We remark that the study of the aforementioned issues may give some insight on the connection between problems that admit distributional formulation with their limiting counterpart without this kind of structure. For a better comprehension on this subject we refer the reader to \cite{BRS}, \cite{CharPar13}, \cite{CharPer07}, \cite{daSR19}, \cite{daSRS19}, \cite{daSS19}, \cite{daSS19-2}, \cite{DelPFR}, \cite{DelPR19}, \cite{DiCKP}, \cite{FerLla}, \cite{JLM99} or \cite{RSdaS18}.

\subsection{Statement of the main results}\label{StemMR}

We address problem \eqref{Eq1} using variational methods. Let us then consider the energy functional $\mathcal{J}_p: W^{1,p}_0(\Omega)\times W^{1,p}_0(\Omega)\to\mathbb{R}$ given by
\begin{equation}\label{11-07}
\displaystyle \mathcal{J}_p(w, z) \defeq \frac{1}{2}\frac{\|\nabla w\|_{\elle p}^p+[\bar  w]_{p,s}^p+\|\nabla z\|_{\elle p}^p+[\bar  z]_{p,r}^p}{\io|w|^\alpha|z|^\beta dx},
\end{equation}
where $\bar  w$ is the extension to $0$ of $w$ in all $\mathbb{R}^N$ and
\[
  [\bar  w]_{p,s}^p=\int\int_{\mathbb{R}^{2N}} \dfrac{|\bar  w(x)-\bar  w(y)|^p}{|x-y|^{N+sp}}dxdy
\]
is the \textit{Gagliardo seminorm} of $w$ (same definitions for $\bar  z$ and $[\bar  z]_{p,s}^p$). Since $[\bar  w]_{p,s}\le C\|\nabla w\|_{\elle p}$ for any $w\in W^{1,p}_0(\Omega)$ (see Lemma \ref{nofrac}), the functional above is well-defined and in the sequel, with a slight abuse of notation, we often write $[w]_{p,s}$ instead of $[\bar  w]_{p,s}$.

Our first step is to show that there exist $(\lambda_p,u_p,v_p)$ such that
\begin{equation}\label{p-Min}\tag{{\bf $p-$Min}}
\begin{array}{rcl}
  \lambda_p & = & \mathcal{J}_p(u_p,v_p) \\
   & = &  \displaystyle  \min\left\{ \mathcal{J}_p (w,z): (w, z) \in \w1p\times\w1p  \quad \text{and}\quad wz\neq0\right\}.
\end{array}
\end{equation}
Thanks to the Lagrange Multipliers Theorem and the definition of $\lambda_p$, it follows that  $\lambda_p$ is the first, smallest variational eigenvalue of \eqref{Eq1} and $(u_p,v_p)$ is the associated eigenpair, namely $(u_p,v_p)$ is a weak solution of \eqref{Eq1} with $\lambda=\lambda_p$ (see Definition \ref{nyc} for the precise definition of weak solution). More in detail we have the following Theorems.

\begin{thm}[{\bf Existence of solutions}]\label{MThm1} Set $p\in(1,\infty)$ and assume \eqref{parass}. Then $\lambda_p$ is reached at $(u_p,v_p)$, with $u_p$ and $v_p$ not changing sign $a.e.$ in $\Omega$. Moreover $\lambda_p$ is the the first smallest variational eigenvalue for problem \eqref{Eq1} and $(u_p,v_p)$ the relative eigenpair.
\end{thm}

\begin{thm}[{\bf Simplicity of the eigenvalue}]\label{MThsimp}
 Under the assumptions of Theorem \ref{MThm1} it holds true that $\lambda_p$ is simple, i.e. if $(\tilde {u}_p,\tilde {v}_p)$ is another eigenpair associated to $\lambda_p$, then there exists a constant $k\neq0$ such that either $(\tilde {u}_p, \tilde {v}_p)=k(u_p, v_p)$ or $(\tilde {u}_p, \tilde {v}_p)=k(u_p, -v_p)$.
\end{thm}

Let us take now three increasing sequences $\pn,\an,\bn\to\infty$ and assume
\begin{equation}\label{pitt}
  \displaystyle \an+\bn=\pn, \qquad \text{and} \qquad \lim_{n\to\infty} \frac{\an}{\pn}=\Gamma \,\,\,\,\, \text{and} \,\,\,\,\, \lim_{n\to\infty}\frac{\bn}{\pn}=1-\Gamma.
\end{equation}
Theorem \ref{MThm1} assures the existence of $(\lambda_{n},u_{p_n},v_{p_n})=(\lambda_{n},\un,\vn)$, solution of \eqref{Eq1} for any $n\in\mathbb{N}$, such that $\un,\vn>0$ $a.e.$ in $\Omega$ and
\[
  \displaystyle \io|\un|^{\an}|\vn|^{\bn} dx=1.
\]
Our next result says that $(\lambda_n)^{\frac{1}{\pn}}$ converges to the first eigenvalue of the limiting functional
\begin{equation}\label{inffunc}
\mathcal{J}_\infty(w, z) \defeq \frac{\max\left\{\|\nabla w\|_{\elle \infty},|w|_{r},\|\nabla z\|_{\elle \infty},|z|_{s}\right\}}{\|w^{\Gamma}z^{1-\Gamma}\|_{L^{\infty}(\Omega)}},
\end{equation}
where $\displaystyle |w|_{r}=\sup_{x,y\in\Omega \atop{x \neq y}}\dys \frac{|w(x)-w(y)|}{|x-y|^{s}}$ (same definition for $|z|_{r}$), and that $(\un,\vn)$ uniformly converge to the relative minimizer.

\begin{thm}[{\bf Limiting minimizers and geometric characterization of $\lambda_{\infty}$}]\label{MThm2} Assuming \eqref{pitt} we have that
\[
    \lambda_\infty \defeq \lim_{n \to \infty} {\left(\lambda_n \right)}^{\frac{1}{p_n}} \\
   =  \max\left\{ \frac{1}{R} , \frac{1}{R^{\Gamma r+(1-\Gamma)s}}\right\},
\]
where $\displaystyle R \defeq \max_{x \in \Omega} \dist(x, \partial \Omega)$. Moreover there exists $(u_{\infty}, v_{\infty}) \in W_0^{1, \infty}(\Omega)\times W_0^{1,\infty}(\Omega)$, such that $\|\ui^{\Gamma}\vi^{\Gamma-1}\|_{\elle{\infty}}=1$, $(u_p, v_p) \to (u_{\infty}, v_{\infty})$ uniformly in $\Omega$ (up to subsequences) and
$$
\begin{array}{rcl}
   \lambda_\infty & = & \mathcal{J}_{\infty}(u_{\infty}, v_{\infty})\\
   &  = &\displaystyle \min\left\{ \mathcal{J}_{\infty} (w,z): (w, z) \in W_0^{1, \infty}(\Omega)\times W_0^{1,\infty}(\Omega)\right\}.
\end{array}
$$
\end{thm}
It is interesting to notice that the geometric characterization of $\lambda_{\infty}$ strongly depends on a simple geometric property of the domain: the radius of the largest ball contained in $\Omega$ (cf. \cite[Theorem 1.2]{DelPFR}). Indeed, if $R<1$, the limiting eigenvalue is neither affected by the presence of the nonlocal diffusion nor by the behaviour of the sequences $\an,\bn$ (and one recover the same result of \cite[Theorem 1.1]{BRS}). On the other hand, if $R>1$, the presence of the two fractional $p-$Laplacians and the limiting ratio $\Gamma \in (0, 1)$ come into the play trough the convex combination $\Gamma r+(1-\Gamma)s$ (cf. \cite[Theorem 1.2]{DelPR19}). The proof of such a dichotomy phenomenon strongly rely on a very specific choice of test functions in the limiting functional \eqref{inffunc}, see p. \pageref{existenceandgeometric} below for details.

In our last result, we show that $(u_{\infty}, v_{\infty})$ is not just the minimizer of the limiting functional \eqref{inffunc}, but it also solves in the viscosity sense a \emph{limiting PDE system}, obtained, in some sense, by passing to the limit as $p\to\infty$ in \eqref{Eq1}.

Before stating this result, we recall the definition of the nowadays well-known \textit{$\infty-$Laplacian}
$$
    \Delta_\infty w(x) \defeq \sum_{i, j=1}^{n} \frac{\partial w}{\partial x_j}(x)\frac{\partial^2 w}{\partial x_j \partial x_i}(x) \frac{\partial w}{\partial x_i}(x) = \left\langle D^2 w(x)\nabla u(x), \nabla u(x)\right\rangle,
$$
and the definition of the \textit{H\"{o}lder $\infty-$Laplacian}
\begin{align}\label{l.infty}
    \L_{\infty, t} w(x) \defeq  \L_{\infty, t}^+ w(x) + \L_{\infty, t}^- w(x),
\end{align}
where
\begin{equation}\label{l.infty.pm}
   \L_{\infty, t}^{+} w(x) \defeq \sup_{y \in \mathbb{R}^N}  \frac{w(x)-w(y)}{|x-y|^{t}} \quad \text{and} \quad \L_{ \infty, t}^{-} w(x) \defeq \inf_{y \in \mathbb{R}^N}  \frac{w(x)-w(y)}{|x-y|^{t}}.
\end{equation}

\begin{thm}[{\bf Limiting PDE system}]\label{MThm4} Suppose that assumptions of Theorem \ref{MThm2} are in force. Then, $(u_{\infty}, v_{\infty})$ is a viscosity solution to
\begin{equation}\label{EqLim}
  \left\{
\begin{array}{rcrcl}
   \max\{\mathrm{G}^r_1[u, v], \, \mathrm{G}^r_2[u, v]\} & = & 0 & \text{in} & \Omega\\
    \max\{\mathrm{G}^s_1[v, u], \, \mathrm{G}^s_2[v, u]\} & = & 0 & \text{in} & \Omega\\
   u & = & 0 & \text{on} & \mathbb{R}^N \setminus \Omega\\
   v & = & 0 & \text{on} & \mathbb{R}^N \setminus \Omega,
\end{array}
\right.
\end{equation}
where, for $t\in(0,1)$,

$$
   \mathrm{G}^t_1[u, v] \defeq \min\left\{\mathcal{L}_{\infty, t} u,\,  \L_{\infty, t}^{+} u - \lambda_{\infty}u^{\Gamma}v^{1-\Gamma},\, \mathcal{L}_{\infty, t}^{+} u -|\nabla u|\right\}
$$
and
$$
   \mathrm{G}^t_2[u, v] \defeq \min\left\{-\Delta_{\infty} u, \,|\nabla u| - \lambda_{\infty}u^{\Gamma}v^{1-\Gamma}, \,|\nabla u| + \mathcal{L}_{\infty, t}^{-} u, \,|\nabla u| -\mathcal{L}_{\infty, t}^{+} u\right\}.
$$
\end{thm}
It is important to point out that besides its own interest, Theorems \ref{MThm2} and \ref{MThm4} provide an alternative mechanism to establish existence of viscosity solutions to \eqref{EqLim}, which is
a non-trivial endeavor, since it is not crystal clear whether the involved operators fulfill a comparison principle or not. For this very reason, existence and uniqueness assertions cannot be established via classical Perron's method.

In order to establish our results, we have to overcome some technical obstacles and adopt certain alternative approaches, which, for the best of our knowledge, have not been put into practice for this kind of problems before (cf. \cite{BRS}, \cite{DelPR19}, \cite{FIN99}, \cite{JLM99} and \cite{LL14}), see Sections \ref{Section2}, \ref{Sec3}, \ref{Sec5} and Appendix \ref{Appendix} for more details.

Finally, let us mention briefly some possible applications of our results, at least in the particular case where formally $\alpha=p,\beta=0$,  i.e., the single equation. For instance Theorem \ref{MThm1} could be employed as a first step to deal with shape optimization problems, like nonlocal Faber-Krahn or Hong-Krahn-Szego type inequalities, see \cite[Theorems 1.3 and 1.4]{GS19}. Moreover, Theorem \ref{MThm1} allows the investigation of existence or non-existence of positive solutions for certain one parameter problems, like
\begin{equation}\label{28-11}
\left\{
\begin{array}{rclcl}
-\Delta_p u + (-\Delta)^s_p u & = &  \lam |u|^{p-2}u + f_{\lambda}(x, u) & \mbox{in} & \Omega \\
  u& =& 0&\text{ on } & \R^n \setminus \Omega,
\end{array}
\right.
\end{equation}
under suitable assumptions on the behaviour of $f_{\lambda}(x, \cdot)$ near the origin and at infinity. Actually, problem \eqref{28-11} is expected to have a continuous branch of solutions $u_{\lambda}$  that bifurcates to infinity as $\lambda$ approaches the associated first eigenvalue. Of particular interest, when $f_{\lambda}$ is a concave or convex power, i.e. $u^{q(p)-1}$, for either $1<q(p)<p$ or $p<q(p)$, we can obtain existence results for both the original and the limit problems as in \cite{CharPer07}.

 Lastly, for the sake of completeness, we also mention that recently, in order to investigate the concave-convex problem, for exponents $q$ and $r$ satisfying $0<q(p)<p-1<r(p)<\infty$:
$$
\left\{
\begin{array}{rclcl}
-\Delta_p u(x) + (-\Delta)^s_p u(x) & = &  \lam_p u^q(x) + u^r(x) & \mbox{in} & \Omega \\
u(x)&>&0&\mbox{in}& \Omega\\
  u(x)& =& 0&\text{ on } & \R^n \setminus \Omega,
\end{array}
\right.
$$
the scalar version of \eqref{Eq1} (Theorem \ref{MThm1}) has played a decisive role in obtaining existence of weak solutions, see \cite[Appendix A]{daSS19-2} for further details.\\

\section{Notations, functional setting and background results}\label{Section2}

In this section we collect all the notations, preliminary results and definitions that we need throughout the article.

\subsection{Functional setting}\label{Appenv}
Let us recall the standard definition of the fractional Sobolev space of exponents $t\in(0,1)$ and $p\in(1,\infty)$ in all $\mathbb{R}^N$ (see \cite{DiNPV} or \cite{leoni2}):
\[
    W^{t,p}(\mathbb{R}^N) \defeq \left\{f \in L^p(\mathbb{R}^N): [f]_{t,p}=\int_{\mathbb{R}^{2N}}\frac{|f(y)-f(x)|^p}{|y-x|^{N+tp}}dxdy<\infty \right\}.
\]
Usually, in order to deal with the $t$-fractional $p-$Laplacian in the bounded domain $\Omega$, one has to consider solutions  for which the zero extension to $\Omega^C$ belongs to $W^{t,p}(\mathbb{R}^N)$. This is due to the nonlocal nature of the operator, which force to consider the contribution of the solutions inside and outside the domain $\Omega$. However, in our case, the simultaneous presence of the local operator ($\Delta_p$) allows us to work in a more simple environment space. Indeed we have the following Lemma.
\begin{lemma}\label{nofrac} For any $p\in(1,\infty)$ and $s\in(0,1)$, there exist a constant $C=C(N,s,\Omega)$ such that $C_p = C_p(N,s, p, \Omega) \to C\in(0,\infty)$ as $p\to\infty$ and
$$
   [\bar w]_{s,p}\le C_p \|\nabla w\|_{\elle p} \ \ \ \forall \ w\in W^{1,p}_0(\Omega),
$$
where $\bar w$ is the extension to $0$ of $w$ in all $\mathbb{R}^N$.
\end{lemma}
\begin{proof}
The proof follows by adapting the argument of \cite[Proposition 2.2]{DiNPV}. We provide some details for the convenience of the reader. For any $w\in W^{1,p}_0(\Omega)$ let $\bar w\in W^{1,p}(\mathbb{R}^N)$ be the extension to $0$ of $w$ in all $\mathbb{R}^N$. We recall that $\|\nabla \bar w\|_{L^{p}(\mathbb{R}^N)}=\|\nabla w\|_{\elle p}$. We have that
\[
   \int_{\mathbb{R}^{2\mathbb{N}}}\frac{|\bar w(x)-\bar w(y)|^p}{|x-y|^{N+ps}}dxdy=2\underbrace{\int_{\Omega\times\Omega^c}\frac{|\bar w(x)-\bar w(y)|^p}{|x-y|^{N+ps}}dxdy}_{I_1}+\underbrace{\int_{\Omega\times\Omega}\frac{|w(x)-w(y)|^p}{|x-y|^{N+ps}}dxdy}_{I_2}.
\]
We provide the explicit computation only for $I_1$, being the treatment of $I_2$ similar. For this purpose, let $B_{\rho}(x)$ be the ball centered at $x\in\Omega$ with $\rho$ the diameter of $\Omega$. We have that
$$
\begin{array}{rcl}
  \displaystyle \int_{\Omega}\int_{\Omega^c\cap B_{\rho}(x)}\frac{|\bar w(x)-\bar w(y)|^p}{|x-y|^{N+ps}}dxdy & \le & \displaystyle \int_{\Omega}\int_{B_{\rho}(x)}\frac{|\bar w(x)-\bar w(y)|^p}{|x-y|^{N+ps}}dxdy\\
   & \le & \displaystyle \io\int_{B_{0,\rho}}\frac{|\bar w(x)-\bar w(x+z)|^p}{|z|^{N+ps}}dxdz \\
   & \le & \displaystyle \io\int_{B_{0,\rho}}\int_0^1\frac{|\nabla\bar w(x+tz)|^p}{|z|^{N+ps-p}}dtdxdz\\
   & \le & \displaystyle   \omega_N \|\nabla \bar w\|^p_{L^p(\mathbb{R}^N)}\int_0^{\rho}r^{p(1-s)-1}\\
   & \le & \displaystyle \frac{\omega_N}{p(1-s)}\rho^{p(1-s)}\|\nabla w\|^p_{\elle p}.
\end{array}
$$
On the other hand, we have that
$$
\begin{array}{rcl}
\displaystyle \int_{\Omega}\int_{\Omega^c\cap B^c_{\rho}(x)}\frac{|\bar w(x)-\bar w(y)|^p}{|x-y|^{N+ps}}dxdy & \le & \displaystyle \int_{\Omega}\int_{B_{\rho}(x)^c}\frac{|\bar w(x)-\bar w(y)|^p}{|x-y|^{N+ps}}dxdy\\
&\le& \displaystyle \frac{\omega_N}{ps}\rho^{-ps}\|w\|_{\elle p}^p\\
&\le& \displaystyle \frac{\omega_N \mbox{diam}(\Omega)^p}{p^2s}\rho^{-ps}\|\nabla w\|_{\elle p}^p,
\end{array}
$$
where in the last line we have used the Poincar\'{e} inequality, see e.g. \cite[Theorem 13.19]{leoni2} .
\end{proof}

\subsection{Weak and viscosity solutions}\label{DefWeaViscSol}

In this work we will deal with different notions of solutions, and for the sake of clarity, in this paragraph we specify their definitions. Indeed, while for fixed values of $1<p< \infty$, we are going to consider weak and viscosity solutions, in the limiting setting, as $p \to \infty$, we will use the notion of viscosity solutions, only.

Before that, let us introduce the following useful notation: for any $w,\psi\in W^{1, p}_0(\Omega)$ we denote
\[
   \E_p(w,\psi)=\int_\Omega |\nabla w|^{p-2}\nabla w\cdot \nabla \psi\,dx
\]
and
\[
   \E_{t,p}(w, \psi)=\int_{\mathbb{R}^{2N}} \frac{|w(x)-w(y)|^{p-2}(w(x)-w(y))(\psi(x)-\psi(y))}{|x-y|^{n+tp}}\,dxdy.
\]
We are now ready to specify the definitions of solutions employed throughout the present article.
\begin{definition}[{\bf Weak solution}]\label{nyc}
A couple $(u, v) \in W^{1, p}_0(\Omega) \times W^{1, p}_0(\Omega)$ is a weak solution to \eqref{Eq1} if for all $\psi, \varphi \in W^{1, p}_0(\Omega)$
it holds
$$
\left\{
\begin{array}{rcl}
  \E_p(u,\psi) + \E_{r,p}(u,\psi) & = & \displaystyle \lam\frac{2\alpha}{\alpha+\beta}\int_\Omega |u|^{\alpha-2}u |v|^{\beta} \psi \,dx  \\
  \E_p(v,\varphi) + \E_{s,p}(v,\varphi)& = & \displaystyle \lam \frac{2\beta}{\alpha+\beta}\int_\Omega |u|^{\alpha} |v|^{\beta-2}v \varphi \,dx . \\
\end{array}
\right.
$$
\end{definition}
In our approach, it is more convenient to use for fixed values of $p$  a notion of decoupled viscosity solution for \eqref{Eq1}, see \cite{BRS} and \cite{DelPR19} for the corresponding definitions in the local and nonlocal cases. Indeed, we consider the couple $(u,v)$ as a viscosity solution for each equation of system \eqref{Eq1}, separately as follows:

\begin{definition}[{\bf Decoupled viscosity solution}]\label{DefViscSol}
 A couple $(u, v) \in C(\Omega) \times C(\Omega)$ is said to be a viscosity subsolution (resp. supersolution) to the first equation of \eqref{Eq1} if, whenever $x_0 \in \Omega$ and $\phi\in C^2(\Omega)\cap C^1_0(\Omega) $, with $\phi(x_0)=u(x_0)$, are such that $u-\phi$ has a strict local maximum (resp. minimum) at $x_0$, then
\begin{equation*}
    -\Delta_p \phi(x_0) +(-\Delta)^r_p \phi(x_0)  \leq  \lam |\phi(x_0)|^{\alpha-2}\phi(x_0)|v(x_0)|^{\beta} \quad (\text{resp.} \,\,\,\geq ....).
\end{equation*}
Respectively, the couple $(u, v)$ is said to be a viscosity subsolution (resp. supersolution) to the second equation of \eqref{Eq1} if, whenever $y_0 \in \Omega$ and $\psi\in C^2(\Omega)\cap C^1_0(\Omega) $, with $\psi(y_0)=v(y_0)$, are such that $v-\psi$ has a strict local maximum (resp. minimum) at $y_0$, then
\begin{equation*}
    -\Delta_p \psi(y_0) +(-\Delta)^s_p \psi(y_0) \leq  \lam |u(y_0)|^{\alpha}|\psi(y_0)|^{\beta-2}\psi(y_0) \quad (\text{resp.} \,\,\,\geq ....)
\end{equation*}

Finally, $(u, v) \in C(\Omega) \times C(\Omega)$ is said to be a viscosity solution to the first equation (resp. the second) of \eqref{Eq1} if it is simultaneously a viscosity subsolution and a viscosity supersolution to the corresponding equation.
\end{definition}
For the limit case, according to the notation introduced in \eqref{EqLim}, some minor modifications in the definition of viscosity solutions have to be considered.
\begin{definition}[{\bf Viscosity solution for the limit equation}]\label{DefVSlimeq} A couple $(u, v) \in C(\Omega) \times C(\Omega)$ is said to be a viscosity subsolution (resp. supersolution) to \eqref{EqLim} if, whenever $x_0 \in \Omega$ and $\phi, \psi \in C^2(\Omega)\cap C^1_0(\Omega)$, with $\phi(x_0)=u(x_0)$ and $\psi(x_0)=v(x_0)$, are such that $u-\phi$ and $v-\psi$ have a strict local maximum (resp. minimum) at $x_0$, then
\begin{equation}\label{16-10bis}
\left\{
\begin{array}{rcl}
   \max\{\mathrm{G}^r_1[\phi(x_0), \psi(x_0)],\mathrm{G}^r_2[\phi(x_0), \psi(x_0)]\} &\le& 0 \quad (\text{resp.} \,\,\,\geq ....)\\
   \max\{\mathrm{G}^s_1[\psi(x_0), \phi(x_0)],\mathrm{G}^s_2[\psi(x_0), \phi(x_0)]\} &\le& 0 \quad (\text{resp.} \,\,\,\geq ....)
\end{array}
\right.
\end{equation}

Accordingly, $(u, v) \in C(\Omega) \times C(\Omega)$ is said to be a viscosity solution to \eqref{EqLim} if it is
simultaneously a viscosity subsolution and a viscosity supersolution.

\end{definition}
For the sake of clarity, let us stress that the {\bf first inequality} in \eqref{16-10bis} means that (see Theorem \ref{MThm4})
\[
\min\left\{\mathcal{L}_{\infty, r} \phi(x_0),\,  \L_{\infty, r}^{+} \phi(x_0) - \lambda_{\infty}\phi(x_0)^{\Gamma}\psi(x_0)^{1-\Gamma},\, \mathcal{L}_{\infty, r}^{+} \phi(x_0) -|\nabla \phi(x_0)|\right\} \le 0
\]
\[\mbox{\textbf{and} (\textbf{or} in the case $\ge$ ) }\]
\[
\min\left\{-\Delta_{\infty} \phi(x_0), \,|\nabla \phi(x_0)| - \lambda_{\infty}\phi(x_0)^{\Gamma}\psi(x_0)^{1-\Gamma}, \,|\nabla \phi(x_0)| + \mathcal{L}_{\infty, r}^{-} \phi(x_0), \,|\nabla \phi(x_0)| -\mathcal{L}_{\infty, r}^{+} \phi(x_0)\right\}\le 0.
\]

The following lemma provides a relation between weak and viscosity sub/supersolutions to the decoupled equations of \eqref{Eq1}, see Definition \ref{DefViscSol}. We refer the reader to \cite{JLM99} and \cite{LL14} for similar results in the local and nonlocal settings.

\begin{lemma}[{\bf Weak solutions are viscosity solutions}]\label{EquivSols} If $(u,v) \in \left(W^{1,p}_0(\Omega)\cap C(\Omega)\right)^2$ is a weak supersolution (resp. subsolution) to \eqref{Eq1} then it is also a decoupled viscosity supersolution (resp. subsolution) to the first and the second equation of \eqref{Eq1}.
\end{lemma}

\begin{proof}Since the case of the subsolution and the analysis for $v$ are similar, let us just prove that $u\in W^{1,p}_0(\Omega)\cap C(\Omega)$ is a  viscosity supersolution to
\[
-\Delta_p u+\left(-\Delta\right)_p^ru=f_{\lambda}(u)
\]
where $f_{\lambda}(u)=\lambda \frac{2\alpha}{\alpha+\beta}|u|^{\alpha-2}u|v|^{\beta}$.   Fix $x_0 \in \Omega$ and $\phi \in C^2(\Omega)$ a test function such that $u(x_0) = \phi(x_0)$ and $u(x)> \phi(x)$ for $x \neq x_0$. Our goal is to show that
$$
    -\Delta_p \phi(x_0)+ \left(-\Delta\right)_p^r\phi(x_0)\geq f_{\lambda}(\phi(x_0)).
$$
Let us suppose, for the sake of contradiction, that the inequality does not hold. Then, by continuity there exists an $r>0$ small enough such that
$$
   -\Delta_p \phi(x)+ \left(-\Delta\right)_p^r\phi(x) < f_{\lambda}(\phi(x)) \quad \text{in} \quad B_r(x_0).
$$
Now, we define the auxiliary function
$$
    \psi(x) \defeq \phi(x)+ \frac{1}{10} \inf_{\partial B_r(x_0)} (u(x)-\phi(x)).
$$
Notice that $\psi$ fulfils $\psi < u$ on $\partial B_r(x_0)$, $\psi(x_0)> u(x_0)$ and
\begin{equation}\label{EqPsi}
     -\Delta_p \psi(x) + \left(-\Delta\right)_p^r\psi(x) < f_{\lambda_p}(\phi(x)) \quad \text{in} \quad B_r(x_0).
\end{equation}
Since the continuous function $\psi-u$ is negative on $\partial B_r(x_0)$, we have that
\[
\varphi:=(\psi-u)^{+} = \max\{\psi-u, 0\}\in W^{1,p}_0(\Omega).
\]
Taking $\varphi$ as a test function in \eqref{Eq1}, we obtain
\begin{equation}\label{Eq3.4}
    \displaystyle \mathcal{E}_p(u,\varphi) +  \mathcal{E}_{r, p}(u, \varphi) = \int_{B_r(x_0)} f_{\lambda}(u(x))\varphi dx.
\end{equation}
On the other hand, from \eqref{EqPsi} we get
\begin{equation}\label{Eq3.5}
    \displaystyle \mathcal{E}_p(\psi, \varphi) + \mathcal{E}_{r, p}(\psi, \varphi) < \int_{B_r(x_0)} f_{\lambda}(\phi(x))\varphi dx.
\end{equation}
Now, by subtracting \eqref{Eq3.4} from \eqref{Eq3.5}, we obtain
{\small{
$$
\begin{array}{rcl}
  I_1+I_2 & = & \displaystyle \int_{\{\psi \ge u\}}(|\nabla u|^{p-2}\nabla u-|\nabla \psi|^{p-2}\nabla\psi)\nabla (\psi-u)
+\iint_{\mathbb{R}^{2N}}\frac{[U(x,y)-\Psi(x,y)](\varphi(x)-\varphi(y))}{|x-y|^{N+rp}}dxdy \\
   & > & \displaystyle \int_{B_r(x_0)}  \left(f_{\lambda_p}(u)-f_{\lambda_p}(\phi)\right)\varphi > 0.
\end{array}
$$}}
where
\[
U(x,y)=|u(x)-u(y)|^{p-2}(u(x)-u(y)) \ \ \ \mbox{and} \ \ \ \Psi(x,y)=|\psi(x)-\psi(y)|^{p-2}(\psi(x)-\psi(y)).
\]
The monotonicity of the $p-$Laplacian implies that $I_1\le0$ and moreover, arguing as in \cite[Lemma 9]{LL14}, we also have $I_2\le0$, that yields to a contradiction.
\end{proof}

\subsection{Others preliminary results}
\label{opr} For the convenience of the reader, in this subsection we collect some auxiliary results which will play decisive roles in the proofs of Theorems \ref{MThsimp} and \ref{MThm4}.

\begin{lemma}[\cite{BK}]\label{16-8}
Let us consider $u,\tilde u\in\w1p$ such that $u,\tilde u> 0$ in $\Omega$ and set $\varphi=\left(\frac{u^p+\tilde u^p}{2}\right)^{\frac1p}$. Then it results
\[
\|\nabla\varphi\|_{\elle p}^p\le \frac12(\|\nabla u\|_{\elle p}^p+\|\nabla\tilde u\|_{\elle p}^p),
\]
with the equality holding if and only if $u=k\tilde u$ for some constant $k>0$.
\end{lemma}

The next two lemmas play a fundamental role in the proof of the simplicity of the eigenvalue. Since we could not find their proofs in the literature, we provide the details for the sake of completeness.
\begin{lemma}\label{bigin} Given $x,y,w,z\ge0$ there holds that
\[ g(x,y,w,z) =\left|(x^p+y^p)^{\frac{1}{p}}-(w^p+z^p)^{\frac{1}{p}}\right|-\left(|x-w|^p+|y-z|^p\right)^{\frac{1}{p}}\leq 0,\]
where $p\geq 1$.
\end{lemma}
\begin{proof}
{Indeed, without loss of generality we may suppose that the largest norm between $x,y,w$ and $z$ is given by $|w|>0$. Then, by considering
\[g(x,y,w,z) =  |w|f\left(\dfrac{|x|}{|w|},\dfrac{|y|}{|w|},\dfrac{|z|}{|w|}\right)\]
we can apply Lemma \ref{lemma1ap} to the function $f$ and conclude that $g$ is a non-negative function.}
\end{proof}

\begin{lemma}\label{18-8}
Let $p>1$ and $\alpha, \beta>0$ such that $\alpha+\beta=p$. Then, for any quadruple of strictly positive real numbers $a,b,c,d$, it holds true that
\[
(a^p+b^p)^{\frac{\alpha}{p}}(c^p+d^p)^{\frac{\beta}{p}}\ge a^{\alpha}c^{\beta}+b^{\alpha}d^{\beta},
\]
whit the equality holding if and only if $(a,c)=k(b,d)$ for some positive constant $k>0$.
\end{lemma}
\begin{proof}
Proving inequality \eqref{18-8} is equivalent to show that
\[
f(x,y)=(1+x^p)^{\frac{\alpha}{p}}(1+y^p)^{\frac{\beta}{p}}-1-x^{\alpha}y^{\beta}\ge0 \ \ \ \forall \ x,y\in(0,\infty)\times(0,\infty)
\]
with the equality satisfied if and only if $x=y$. We have that $f(x,x)=0$ and moreover
$$
\begin{array}{rcl}
  \frac{\partial}{\partial x} f(x,y) & = & \alpha(1+x^p)^{\frac{\alpha}{p}-1}(1+y^{p})^{\frac{\beta}{p}}x^{p-1}-\alpha x^{\alpha-1}y^{\beta} \\
   & = & \alpha(1+y^{p})^{\frac{\beta}{p}}x^{\alpha-1}\left[\frac{x^{\beta}}{(1+x^p)^{\frac{\beta}{p}}}-\frac{y^{\beta}}{(1+y^p)^{\frac{\beta}{p}}}\right].
\end{array}
$$

Thus we have that
\[
\frac{\partial}{\partial x} f(x,y)<0 \ \ \mbox{for} \ x<y \ \ \ \mbox{and} \ \ \ \frac{\partial}{\partial x} f(x,y)>0 \ \ \mbox{for} \ x>y
\]
and the proof is concluded.
\end{proof}

As we anticipate in the introduction, the H\"{o}lder infinite Laplace operator $\L_{\infty,t}$, defined in \eqref{l.infty}--\eqref{l.infty.pm}, plays an essential role in describing the limiting problem. Particularly, the following Lemma relates the fractional $p-$Laplacian with $\L_{t, \infty}^{\pm}$.

In order to be self contained,  we provide an alternative proof for the next lemma. The interested reader can have a look at different proofs, see e.g. {\cite[Lemma 6.5]{CLM} and \cite[Lemma 6.1]{FerLla}}.

\begin{lemma}\label{Lemlim-op} Let $\varphi  \in C_0^1(\overline{\Omega})$ be a test function (extended by zero outside $\Omega$) and $x_n \to x$ as $p_n  \infty$. Then, given $ \{\varphi_n\}\subset C_0^1(\overline{\Omega})$ such that $\varphi_n \to \varphi$ uniformly in $\Omega$ there holds that
$$
   \L_{p_n, t}^+(\varphi_n(x_p)) \to \L_{\infty, t}^{+}\varphi(x_0) \quad \text{and} \quad \L_{p_n, t}^-(\varphi_n(x_p)) \to -\L_{\infty, t}^{-}\varphi(x_0)
$$
where
\begin{equation}\label{18-10}
\left\{
\begin{array}{rcl}
   \displaystyle (\L_{p_n, t}^+)^{p_n-1}\varphi(x) & \defeq& 2 \displaystyle \int_{\mathbb{R}^N} \frac{|\varphi(x)-\varphi(y)|^{p_n-2}}{|x-y|^{N+tp_n}}(\varphi(x)-\varphi(y))^{+}dy\\
   \displaystyle (\L_{p_n, t}^-)^{p_n-1} \varphi(x) & \defeq& \displaystyle 2 \int_{\mathbb{R}^N} \frac{|\varphi(x)-\varphi(y)|^{p_n-2}}{|x-y|^{N+tp_n}}(\varphi(x)-\varphi(y))^{-}dy.
\end{array}
\right.
\end{equation}
and
\[
(-\Delta)_{p_n}^t\varphi(x)=(\L_{p_n, t}^+)^{p_n-1}\varphi(x) - (\L_{p_n, t}^-)^{p_n-1} \varphi(x)
\]
\end{lemma}

\begin{proof}
Set $\{f_n^\pm\}\subset C(\overline{\Omega})$ and $f^\pm \in C(\overline{\Omega})$ given by
\begin{equation}
\nonumber
f_n^\pm(y)=\begin{cases}
\dfrac{(\varphi(x_n)-\varphi(y))^\pm}{|x_n-y|^{\frac{N+tp_n}{p_n-1}}} &\mbox{ if }y\neq x_n \\
0 &\mbox{ as } y=x_n
\end{cases}
\end{equation}
and
\begin{equation}
\nonumber
f^\pm(y)=\begin{cases}
\dfrac{(\varphi(x_0)-\varphi(y))^\pm}{|x_0-y|^{t}} &\mbox{ if }y\neq x_0 \\
0 &\mbox{ if } y=x_0.
\end{cases}
\end{equation}

It is obvious that
\begin{equation}
\label{1434}
f_n^\pm \longrightarrow f^\pm \mbox{ uniformly in } \Omega \mbox{ as } n \to \infty.
\end{equation}
Indeed, suppose by contradiction that there exist $\epsilon_0>0$, $\{y_{n_k}\}\subset \Omega$ and $f_{n_k}^\pm$ such that
\[ |f^\pm_{n_k}(y_{n_k})-f^\pm(y_{n_k})|\geq \epsilon_0 \ \forall n \in \mathbb{N}.\]
By using that $\varphi \in C^1_0(\overline{\Omega})$ and that $t \in (0,1)$, the last inequality clearly leads us into a contradiction.

Now, remark that
\begin{equation}
\label{1439}
\left( \int_{\Omega^c} \left(f^\pm_n(y)\right)^{p_n-1}) dy\right)^{\frac{1}{p_n-1}} \longrightarrow \dfrac{(\varphi(x_0))^\pm}{d(x_0,\partial \Omega)^t} \mbox{ as } n \to \infty.
\end{equation}

Indeed, it is enough to observe that
\begin{align*}
\left( \int_{\Omega^c} \left(f^{\pm}_n(y)\right)^{p_n-1} dy\right)^{\frac{1}{p_n-1}}&= \left(2\left|\partial B(0,1)\right|\right)^\frac{1}{p_n-1}(\varphi(x_0))^\pm\left(\int_{d(x_n,\partial \Omega)}^\infty \rho^{-tp_n+1} d\rho\right)^\frac{1}{p_n-1}\\
&=c_n \dfrac{(\varphi(x_0))^\pm}{d(x_0,\partial \Omega)^{\frac{t p_n }{p_n-1}}}
\end{align*}
where \[c_n = \dfrac{\left(2\left|\partial B(0,1)\right|\right)^\frac{1}{p_n-1}}{\left(tp_n\right)^{\frac{1}{p_n-1}}}\to 1 \mbox{ as } n\to \infty.
\]

Thence, by Lemma \ref{1501}
\[\lim_{n\to\infty}\mathcal{L}^{\pm}_{p_n,t} \varphi (x_n) =\max\left\{\left\|\dfrac{(\varphi(x_0)-\varphi(y))^\pm}{|x_0-y|^t}\right\|_{L^\infty(\Omega)}, \dfrac{\varphi(x_0)^\pm}{d(x_0,\partial \Omega)^t}\right\}\]
However, since it is clear that
\[\max\left\{\left\|\dfrac{(\varphi(x_0)-\varphi(y))^+}{|x_0-y|^t}\right\|_{L^\infty(\Omega)}, \dfrac{\varphi(x_0)^+}{d(x_0,\partial \Omega)^t}\right\}=\sup_{y\in \mathbb{R}^N} \dfrac{\varphi(x_0)-\varphi(y)}{|x_0-y|^t}\]
and
\[\min\left\{\left\|\dfrac{(\varphi(x_0)-\varphi(y))^-}{|x_0-y|^t}\right\|_{L^\infty(\Omega)}, \dfrac{\varphi(x_0)^-}{d(x_0,\partial \Omega)^t}\right\}=-\inf_{y\in \mathbb{R}^N} \dfrac{\varphi(x_0)-\varphi(y)}{|x_0-y|^t},\]
the result follows.
\end{proof}

\subsection{Strong Maximum Principle}
\label{Maximum Principle}
In this section, we provide a version of the Strong Maximum Principle, which is obtained using ideas inspired by the works \cite{BrP} and \cite{DiCKP}.

\begin{prop}\label{smp}
Let $\Omega\subset\mathbb{R}^N$ be an open bounded domain, $s\in(0,1)$ and $p\in(1,\infty)$. If $v\in\w1p$ is such that $v\ge0$ in $\Omega$ and
\begin{equation} \label{14:43}
\io|\nabla v|^{p-2}\nabla v\nabla\phi+\int\int_{\mathbb{R}^{2N}}\frac{|v(x)-v(y)|^{p-2}(v(x)-v(y))(\phi(x)-\phi(y))}{|x-y|^{N+ps}}\ge0
\end{equation}
for all $\phi\in \w1p$ with $\phi\ge0$, then either $v=0$ or $v>0$ almost everywhere in $\Omega$.
\end{prop}
\begin{proof}
Firstly, we claim that if $v\equiv 0$ a.e. in a ball $B_R(x_0)\subset \Omega$, than $v\equiv0$ a.e. in $\Omega$.
Indeed,  let us take $0\le\varphi\in C^{\infty}_c(B_R(x_0))$ as a test function in \eqref{14:43}. We have that
$$
\begin{array}{rcl}
  0 & \leq & \displaystyle \underbrace{\io|\nabla v|^{p-2}\nabla v\nabla\varphi}_{=0}+\int\int_{\mathbb{R}^{2N}}\frac{|v(x)-v(y)|^{p-2}(v(x)-v(y))(\varphi(x)-\varphi(y))}{|x-y|^{N+ps}} \\
   & = & \displaystyle \int_{B_R(x_0)}\int_{B_R(x_0)^c}\frac{v(x)^{p-1}\varphi(y)}{|x-y|^{N+ps}}
    \le  \displaystyle -\frac{2}{\mbox{diam}(\Omega)^{N+ps}}\int_{B_R(x_0)^c}v(x)^{p-1}\int_{B_R(x_0)}\varphi(y),
\end{array}
$$
namely $v\equiv0$ a.e. in $B^c_R$.

Now, in order to conclude it is enough prove that if $v(x)=0$ in a set of positive measure, then there exists $
x_0 \in \Omega$ and $R>0$, such that $B_R(x_0)\subset \Omega$ for which $v\equiv 0$ in $B_R(x_0)$. Hence, let us assume that $v(x)=0$ for all $x\in\omega$, where $\omega\subset\Omega$ is a measurable set with $|\omega|>0$. Then there exist $B_{2R}(x_0)\subset\Omega$ such that $|\omega\cap B_R(x_0)|>0$. Take hence $\phi=\frac{\psi^p}{(\delta+v)^{p-1}}$ as test function in \eqref{14:43}, where $\delta>0$ and $\psi\in C^{\infty}_c(B_{2R}(x_0))$ such that $\psi\equiv1$ in $B_R(x_0)$. It results that
\begin{equation}\label{15:24}
(p-1)\io\frac{|\nabla v|^p}{(\delta+v)^{p}}\psi^p\le H(v,\psi,\delta)+p\io\frac{|\nabla v|^{p-1}|\nabla\psi|}{(\delta+v)^{p-1}}\psi^{p-1},
\end{equation}
where
\[
H(v,\psi,\delta)=\int\int_{\mathbb{R}^{2N}}\frac{|v(x)-v(y)|^{p-2}(v(x)-v(y))}{|x-y|^{N+ps}}\left(\frac{\psi^p(x)}{(\delta+v(x))}-\frac{\psi^p(y)}{(\delta+v(y))}\right).
\]
The logarithmic Lemma of \cite[Lemma 1.3]{DiCKP}, applied to our case, assures us that there exists a constant $C=C(p)$, that does not depend on $\delta$, such that
\begin{align*}
H(v,\psi,\delta)\le &C(p) \left(\int\int_{\mathbb{R}^{2N}}\frac{|\psi(x)-\psi(y)|^{p}}{|x-y|^{N+ps}}+R^{N-sp}\right)
\end{align*}
Moreover, using Young inequality in the second term of the right hand side of \eqref{15:24}, we get
\[
p\io\frac{|\nabla v|^{p-1}|\nabla\psi|}{(\delta+v)^{p-1}}\psi^{p-1}\le \frac{p-1}{2}\io\frac{|\nabla v|^p}{(\delta+v)^{p}}\psi^p+\left(\frac{2}{p-1}\right)^{p-1}\io|\nabla \psi|^p.
\]
Plugging these two pieces of information in \eqref{15:24}, it follows that
\[
\int_{B_R(x_0)}\left|\nabla\log\left(1+\frac{v}{\delta}\right)\right|^p=\int_{B_R(x_0)}\frac{|\nabla v|^p}{(\delta+v)^{p}} \le \bar C(p,\psi).
\]
Recalling that $|\omega\cap B_{R}|>0$, we can apply Poincar\'e (see for example \cite{leoni2}) and Chebyshev inequality to get
\[
\log^p\left(1+\frac{t}{\delta}\right)|\{|v|>t\}\cap B_R(x_0)|\le \int_{B_R(x_0)}\log^p\left(1+\frac{v}{\delta}\right)\le C.
\]
Now, if $|\{|v|>t\}\cap B_R(x_0)|=0$ for all $t>0$, then $v\equiv0$ in $B_R(x_0)$. Otherwise there exists $t^*$ such that $|\{|v|>t^*\}\cap B_R(x_0)|>0$. Hence we take the limit as $\delta\to0$ and obtain a contradiction; then once again $v\equiv0$ in $B_R(x_0)$.
\end{proof}

\section{Existence of weak solutions and simplicity eigenvalues}\label{Sec3}

This section is devoted to prove Theorem \ref{MThm1}. For the sake of simplicity, we are going to drop the dependence of the couple $(u,v)$ on $p$.

Firstly, let us consider the functionals $\mathcal{I},\mathcal{G}:W^{1,p}_0(\Omega)\times W^{1,p}_0(\Omega)\to \mathbb{R}$ given by
$$
   \mathcal{I}(w,z) = \int_{\Omega} \left(|\nabla w(x)|^p+|\nabla z(x)|^p\right)dx +  \int\int_{\mathbb{R}^{2N}} \left(\dfrac{| w(x)-w(y)|^p}{|x-y|^{N+rp}}+\dfrac {|z(x)-z(y)|^p}{|x-y|^{N+sp}}\right)dxdy
$$
and
$$
  \mathcal{G}(u,v)=2\int_\Omega |w(x)|^\alpha |z(x)|^{\beta} dx.
$$
Let us recall that $\mathcal{I}$ is well defined thanks to Lemma \ref{nofrac}. Moreover both $\mathcal{I}$ and $\mathcal{G}$ are of class $C^1$ and their Gateaux derivatives are
\begin{equation}\label{gat1}
   \nabla \mathcal{I}(w,z)(\phi,\psi)=p\big( \E_p(w,\phi)+\E_{p,r}(w,\phi), \E_{p,s}(z,\psi)+\E_p(z,\psi)\big)
\end{equation}
and
\begin{equation}\label{gat2}
   \nabla \mathcal{G}(w,z)(\phi,\psi)= 2\bigg(\alpha\int_\Omega |w(x)|^{\alpha-2}|z(x)|^\beta w(x)\phi(x)dx,\beta\int_\Omega |w(x)|^{\alpha}|z(x)|^{\beta-2} z(x)\psi(x)dx\bigg).
\end{equation}
Of course by definition
\[\mathcal{J}_p(u,v)=\dfrac{\mathcal{I}(u,v)}{\mathcal{G}(u,v)}.\]
Now we are in the position to prove the Theorem \ref{MThm1}.
\begin{proof}[{\bf Proof of Theorem \ref{MThm1}}]
Let us prove that there exists a non trivial couple \((u,v)\in W^{1,p}_0(\Omega)\times W^{1,p}_0(\Omega)\) such that
\[
I(u,v)=\min \left\{ \mathcal{I}(w,z):  (w,z)\in W^{1,p}_0(\Omega)\times W^{1,p}_0(\Omega) \,\,\,\text{with}\,\,\, \mathcal{G}(w,z)=1\right\}.
\]
Indeed, observe that for any $(w,z)\in W^{1,p}_0(\Omega)\times W^{1,p}_0(\Omega)$ with $\mathcal{G}(w,z)=1$
\begin{equation}\label{23-9}
\mathcal{I}(u,v)\ge \|\nabla u\|_{\elle p}^p+\|\nabla v\|_{\elle p}^p \ge \mu>0
\end{equation}
where $\mu$ is the first eigenvalue of the system
\begin{equation*}
\left\{
\begin{array}{rclcl}
-\Delta_p \varphi & = & \frac{2\alpha}{\alpha+\beta}\mu |\varphi|^{\alpha-2}\varphi|\psi|^{\beta} & \mbox{in} & \Omega \\
-\Delta_p \psi& = & \frac{2\beta}{\alpha+\beta}\mu |u|^{\alpha}|v|^{\beta-2}v & \mbox{in} & \Omega \\
  u& =& 0&\text{ on } & \partial \Omega\\
  v& =& 0&\text{ on } & \partial \Omega.
\end{array}
\right.
\end{equation*}
From \eqref{23-9} we deduce that our functional is bounded from below and coercive on the closed subset $\mathcal{G}(u,v)=1$. Let us set
\begin{align*}
\lambda_p &:= \inf \left \{ \mathcal{I}(w,z): (w,z)\in W^{1,p}_0(\Omega)\times W^{1,p}_0(\Omega) \mbox{ where }\mathcal{G}(w,z)=1 \right\}>0.
\end{align*}

Thus every minimizing sequence is bounded, i.e., there exists a universal constant $C>0$, such that given \[\{(w_n,z_n)\}\in W^{1,p}_0(\Omega)\times W^{1,p}_0(\Omega),\] satisfying
\[\mathcal{G}(w_n,z_n)=1 \mbox{ and } \lim_{n\to +\infty} \mathcal{I}(w_n,z_n)=\lambda_p,\]
there holds that
$$
   \|w_n\|^p_{W^{1,p}_0}+\|z_n\|^p_{W^{1,p}_0}\leq C \quad  \forall n\in \mathbb{N}.
$$
Let us consider $(u,v) \in W^{1,p}_0(\Omega)\times W^{1,p}_0(\Omega)$ so that, up to subsequences,
\begin{align*}
w_n &\rightharpoonup u,  z_n \rightharpoonup v \quad \mbox{ in } \quad W^{1,p}_0(\Omega)\\
w_n &\rightarrow u,  z_n \rightarrow v \quad \mbox{ in } \quad L^{q}(\Omega), \quad \forall \, q \in [1, p^*)\\
w_n &\rightarrow u,  z_n \rightarrow v \quad \mbox{ a.e. in } \quad \Omega.
\end{align*}
It is straightforward to check that
\[
   \mathcal{G}(u,v) = 1 \quad \mbox{ and } \quad \lambda_p= \lim_{n\to +\infty} \mathcal{I}(u_n,v_n) \geq \mathcal{I}(u,v),
\]
and hence $\lambda_p=\mathcal{I}(u,v)$.

Moreover, since both $\mathcal{I}$ and $\mathcal{G}$ are of class $C^1$, thanks to the Lagrange Multiplier theorem and the definition of $\lambda_p$, we deduce that
\begin{equation}\label{lagrange}
   \nabla I(u,v)(\phi,\psi) = \dfrac{\lambda_p}{p} \nabla G(u,v)(\phi,\psi), \ \ \ \forall \  (\phi,\psi)\in W^{1,p}_0(\Omega)\times W^{1,p}_0(\Omega),
\end{equation}
that is a weak solution of \eqref{Eq1}.

Now we prove that $u$ and $v$ do not change sing on $\Omega$ and that they cannot be zero on a set of positive measure. It is not restrictive to assume that $u,v>0$ on a subset of $\Omega$ with positive measure ($u,v$ are not trivial and one can change $u$ to $-u$ or $v$ to $-v$). Hence we claim that $u,v>0$ a.e. in $\Omega$. Notice at first that the strict inequality
\[
||a|-|b||<|a-b| \ \ \ \forall  \ ab<0,
\]
implies that $u,v\ge0$ a.e. in $\Omega$. Indeed if not
\[
[|u|]_{p,r}^p<[u]_{p,r}^p \ \ \ \mbox{or} \ \ \ [|v|]_{p,r}^p<[v]_{p,r}^p,
\]
that contradicts the minimality of $(u,v)$. At this point it is enough to use Proposition \ref{smp} to prove that claim.
\end{proof}
Now we show that $\lambda_p$ is simple, namely if $(u,v)$ and $(\tilde u, \tilde v)$ are two pairs of solutions of \eqref{Eq1} with $\lambda=\lambda_p$, then there exists $k\in\mathbb{R}$ such that $(\tilde u, \tilde v)=k(u,\pm v)$.
\begin{proof}[{\bf Proof of Theorem \ref{MThsimp}}]
 We restrict our analysis to the case $u,v>0$ and $k\in(0,\infty)$. Indeed all the other configurations can be recovered changing $u$ to $-u$ or $v$ to $-v$ or $k$ to $-k$.\\ First of all let us prove that any non trivial weak solution $(u,v)\in\w1p\times\w1p$ of \eqref{Eq0} with $\lam=\lam_p$ is a minimizer of the functional \eqref{11-07}. For it, let us take $u$ and $v$ as test functions in the first and second equation respectively. We obtain
\[
\|\nabla u\|_{\elle p}^p+[u]_{r,p}^{p}=\frac{2\alpha}{\alpha+\beta}\lambda_p|u|^{\alpha}|v|^{\beta}
\]
and
\[
\|\nabla v\|_{\elle p}^p+[v]_{r,p}^{p}=\frac{2\beta}{\alpha+\beta}\lambda_p|u|^{\alpha}|v|^{\beta}.
\]
Summing up we recover that
\[
\mathcal{J}_p(u,v)=\lambda_p\le \mathcal{J}_p( z,w) \ \ \ \mbox{with} \ \ \ ( z,w)\in \w1p\times\w1p \ \ \ \mbox{and} \ \ \  z\neq 0\neq  w.
\]
Hence if there exist two couple $(u,v)$ and $(\tilde u,\tilde v)$ of solutions of our system, they are also minimizers of \eqref{11-07}. We claim that there exists a constant $k$ such that $(u,v)=k(\tilde u,\tilde v)$. For it, remember that we have $u,v,\tilde{u},\tilde{v}>0$ and  let us define
\[
\varphi=\left(\frac{u^p+\tilde u^p}{2}\right)^{\frac1p} \ \ \ \mbox{and} \ \ \ \ \psi=\left(\frac{v^p+\tilde v^p}{2}\right)^{\frac1p}.
\]
Thanks to Lemmas \ref{16-8} and \ref{bigin} respectively, we have that
\[
\|\nabla \varphi\|_{\elle p}^p\le \frac{1}{2}(\|\nabla u\|_{\elle p}^p+\|\nabla \tilde u\|_{\elle p}^p)
\]
and
\[
[\varphi]_{p,r}^p\le\frac12([u]_{p,r}^p+[\tilde u]_{p,r}^p).
\]
Since the very same holds true for $\psi$ we deduce that

\begin{equation}\label{17:13}
\lambda_p\le \mathcal{J}_{p}(\varphi,\psi) \le \frac12\frac{\mathcal{I}(u,v)+\mathcal{I}(\tilde u,\tilde v)}{\mathcal{G}(\varphi,\psi)}=\frac{\lambda_p}2\frac{\mathcal{G}(u,v)+\mathcal{G}(\tilde u,\tilde v)}{\mathcal{G}(\varphi,\psi)}.
\end{equation}
Moreover, thanks to Lemma \ref{18-8} we deduce that
\begin{equation}\label{17:13bis}
\begin{array}{rcl}
 \displaystyle \mathcal{G}(\varphi,\psi)=2\io\varphi^\alpha\psi^\beta & = & \displaystyle \io(u^p+\tilde u^p)^{\frac{\alpha}{p}}(v^p+\tilde v^p)^{\frac{\beta}{p}} \\
   & \ge & \displaystyle \io u^\alpha v^\beta+\io \tilde u^\alpha \tilde v^\beta\\
   &  = & \frac{\mathcal{G}(u,v)+\mathcal{G}(\tilde u,\tilde v)}{2}.
\end{array}
\end{equation}
In order to avoid a contradiction, both of the inequalities \eqref{17:13} and \eqref{17:13bis} cannot be strict, i.e.  we need that both the inequalities of Lemmas \ref{16-8} and \ref{18-8} have to be satisfied with the equality sign. By Lemma \ref{16-8}, applied to $u$ and $v$, we deduce that there exist two constants $k_1,k_2>0$ for which $u=k_1\tilde u$ and $v=k_2\tilde v$. On the other hand, Lemma \ref{18-8} guarantees that $(u,v)=k(x)(\tilde u,\tilde v)$, for a function  $k(x)>0$ a.e. in $\Omega$. Hence $k(x)=k_1=k_2$ so that $\lambda_p$ is simple according to our definition.

\end{proof}

\section{Existence and geometric characterization to $\infty-$eigenvalue}
\label{existenceandgeometric}

We are in a position to present the proof of the Theorem \ref{MThm2}.

\begin{proof}[{\bf Proof of Theorem \ref{MThm2}}] We follow some ideas of \cite{BRS}, \cite{DelPR19},  \cite{JLM99} and \cite{LL14} adapted to our case. As we already said in the introduction, thanks to assumption \eqref{pitt} and Theorems \ref{MThm1} and \ref{MThsimp}, for any $n\in\mathbb{N}$ there exists a triple $(\lambda_{n},u_{p_n},v_{p_n})=(\lambda_{n},\un,\vn)$, solution of \eqref{Eq1}, such that $\un,\vn>0$ $a.e.$ in $\Omega$ and
\[
\displaystyle \io|\un|^{\an}|\vn|^{\bn} dx=1.
\]
Set $\displaystyle R \defeq \max_{x \in \Omega} \dist(x, \partial \Omega)$, let $B_{R}$ be a ball with radius $R$ contained in $\Omega$ and let us define
\[
d_{R}(x)=\begin{cases}\frac{1}{R}\mbox{dist}(x,\partial B_{R})   &\mbox{if} \ x\in B_{R}\\
0 &\mbox{if} \ x\in \Omega\setminus B_{R}.
\end{cases}
\]
Let us recall the following properties of this normalized distance function
\[
\|d_{R}\|_{\elle{\infty}}=1, \ \ \ |d_{R}|_{t}=R^{-t}, \ \ \ \mbox{and} \ \ \ \|\nabla d_{R}\|_{\elle{\infty}}=R^{-1}.
\]
Now, if $R\leq 1$, let us take $w=z=d_{R}$ as test function in \eqref{inffunc} in order to obtain
$$
\begin{array}{rcl}
  \displaystyle \limsup_{n\to\infty}  \left(\lambda_n\right)^{\frac1{\pn}}  & \le  & \displaystyle \frac{\max\{\|\nabla d_{R}\|_{\elle{\infty}}, |d_{R}|_{r},|d_{R}|_{s}\}}{\|d_{R}\|_{\elle{\infty}}} \\
   \displaystyle & = & \displaystyle \max\left\{\frac{1}{R},\frac1{R^r},\frac1{R^s}\right\} \\
  \displaystyle  & = & \displaystyle \frac{1}{R}.
\end{array}
$$
On the other hand, if $R>1$, choosing
$$
  w_{R}(x)=R^{(1-\Gamma)(r-s)}d_{R}(x) \quad  \text{and} \quad  z_{R}(x)=R^{-\Gamma(r-s)}d_{R}(x)
$$
as test functions in \eqref{inffunc}, it follows
$$
\begin{array}{rcl}
  \displaystyle \limsup_{n\to\infty} \left(\lambda_n\right)^{\frac1{\pn}} & \leq & \displaystyle \frac{\max\{\|\nabla w_{R}\|_{\elle{\infty}}, |w_{R}|_{r},\|\nabla z_{R}\|_{\elle{\infty}},|z_{R}|_{s}\}}{\|w_{R}^{\Gamma}z_{R}^{1-\Gamma}\|_{\elle{\infty}}} \\
   & = & \displaystyle \max\left\{{R}^{(1-\Gamma)(r-s)-1},{R}^{-r\Gamma-s(1-\Gamma)},{R}^{-\Gamma(r-s)-1},{R}^{-r\Gamma-s(1-\Gamma)}\right\} \\
   & = &\displaystyle  \frac{1}{R^{r\Gamma+s(1-\Gamma)}}.
\end{array}
$$
In conclusion, we have that
\begin{equation}\label{firstl}
\limsup_{n\to\infty} \left(\lambda_n\right)^{\frac1{\pn}} \le\max\left\{\frac{1}{R},\frac1{R^{r\Gamma+s(1-\Gamma)}}\right\}.
\end{equation}
The estimate above, the definition of $\lambda_{n}$ and H\"{o}lder Inequality imply that, for a fixed $ q>N$, there exists $n_q$ such that
\begin{equation}\label{24-07}
\|\nabla\un\|_{\elle{q}}+\|\nabla\vn\|_{\elle{q}}\le |\Omega|^{\frac1q-\frac1{\pn}} C \le C \ \ \ \forall n\ge n_q.
\end{equation}
Hence there exists two weak limits $\ui,\vi\in\w{1}{q}$ and two subsequences always labeled with $\{\un\}$ and $\{\vn\}$ such that
\[
\un\rightharpoonup\ui \ \ \ \mbox{and} \ \ \ \vn\rightharpoonup\vi \ \ \ \mbox{in} \ \ \ \w1q,
\]
and, thanks to the choice of $q$ and the classical Sobolev embedding, also
\[
\un\to\ui \ \ \ \mbox{and} \ \ \ \vn\to\vi \ \ \ \mbox{uniformly in } \Omega.
\]
Notice now that the extracted subsequences $\{\un\}$ and $\{\vn\}$ always satisfy \eqref{24-07}, we deduce that for any $\bar q$ there exists $n_{\bar q}$ such that $\{\un\}_{n>n_{\bar q}}$ and $\{\vn\}_{n>n_{\bar q}}$ are uniformly bounded in $W^{1,\bar q}_0(\Omega)$ with respect to $\bar q$. Since they already uniformly converge to $u_{\infty}$ and $v_{\infty}$, we deduce that
\[
\|\nabla u_{\infty}\|_{\elle{\bar q}}+\|\nabla v_{\infty}\|_{\elle{\bar q}}\le C \ \ \ \forall  \ \bar q>N,
\]
that in turn implies $\ui,\vi\in\w{1}{\infty}$. Moreover, thanks to the normalized condition $\io\un^{\an}\vn^{\bn}=1$, it follows that
\[
\|\ui^{\Gamma}\vi^{1-\Gamma}\|_{\elle q}=\lim_{n\to\infty}\left(\io\un^{\frac{\an q}{\pn}}\vn^{\frac{\bn q}{\pn}} dx\right)^{\frac1q}\le \lim_{n\to\infty}|\Omega|^{\frac{1}{q}-\frac1p}=1,
\]
and that
\[
1=\left(\io\un^{\an}\vn^{\bn} dx\right)^{\frac1\pn}\le \|\un^{\an}\vn^{\bn}\|^{\frac{1}{\pn}}_{\elle{\infty}}|\Omega|^{\frac{1}{q}-\frac1{\pn}}.
\]
Hence, by taking the limit both as $n$ and $q$ diverge to  $\infty$, we deduce
$$
   \|\ui^{\Gamma}\vi^{1-\Gamma}\|_{\elle {\infty}}=1.
$$

In order to show that $(\ui,\vi)$ is a minimizer of the the limit functional $\mathcal{J}_{\infty}$, it is useful to consider the following fractional seminorm
\begin{equation}\label{auxilia}
|f|_{t,p}=\left(\int_{\Omega\times\Omega}\frac{|f(x)-f(y)|^p}{|x-y|^{pt}}dxdy\right)^{\frac1p}  \ \ \ \mbox{with}  \ \ \ t\in(0,1) , \ \ p\in(1,\infty).
\end{equation}
Of course if $f\in C^{0,t}(\bar \Omega)$ then $|f|_{t,p}\to |f|_t=\displaystyle \sup_{(x,y)\in\Omega\times\Omega \atop{x \neq y}}\dys \frac{|f(x)-f(y)|}{|x-y|^{t}}$ as $p\to\infty$. At this point notice that for $q<p_n$
\begin{equation}\label{cip}
  \|\nabla\un\|_{\elle q} \le  |\Omega|^{\frac1q-\frac1{\pn}}\|\nabla\un\|_{\elle{\pn}} \le |\Omega|^{\frac1q-\frac1{\pn}}\left(\lambda_n\right)^{\frac1{\pn}}
\end{equation}
and
\begin{equation}\label{ciop}
\begin{array}{rcl}
  |\un|_{r,q}    & \leq & \displaystyle |\Omega|^{\frac2q-\frac2{\pn}}|\un|_{r,\pn} \\
   & \leq  & |\Omega|^{\frac2q-\frac2{\pn}}\mbox{diam}(\Omega)^{\frac{N}{\pn}}[\un]_{r,\pn} \\
   & \leq & |\Omega|^{\frac2q-\frac2{\pn}}\mbox{diam}(\Omega)^{\frac{N}{\pn}} \left(\lambda_n\right)^{\frac1{\pn}},
\end{array}
\end{equation}
where we have used Holder Inequality and the definition of $|\cdot|_{r,q}$ in \eqref{auxilia}. Since the same estimates hold true for $\{\vn\}$ with $s$ instead of $r$, we get that
{\small{
\begin{multline}\label{cipciop}
\left(\|\nabla\un\|_{\elle q}^q+|\un|_{r,q}^q+\|\nabla\vn\|_{\elle q}^q+|\vn|_{s,q}^q\right)^\frac1q \le4^{\frac1q}
|\Omega|^{\frac1q+\frac1{\pn}}\left(|\Omega|^{1+\frac q{\pn}}+\mbox{diam}(\Omega)^{\frac{Nq}{\pn}}\right)^{\frac1q}\left(\lambda_n\right)^{\frac1{\pn}}.
\end{multline}}}
From this inequality we obtain two important pieces of information. The first one is
\begin{equation}\label{28-6-bis}
  \mathcal{J}_{\infty}(u_{\infty}, v_{\infty}) \le  \displaystyle \limsup_{n\to\infty} \left(\lambda_n\right)^{\frac1{\pn}}
    \le  \max\left\{\frac{1}{R},\frac{1}{R^{r\Gamma+s(1-\Gamma)}}\right\},
\end{equation}
that is obtained passing to the limit in \eqref{cipciop} at first as $n\to\infty$ and then as $q\to\infty$. The second piece of information is that, for any $w,z\in W^{1,\infty}_0(\Omega)$, with $\|w^{\Gamma}z^{1-\Gamma}\|_{\elle{\infty}}=1$, it hold true that
\[
\mathcal{J}_{\infty}(u_{\infty}, v_{\infty}) \le \mathcal{J}_{\infty}(w, z),
\]
that is obtained plugging in \eqref{cipciop} the variational characterization of $\lambda_n$ and passing to the limit in $n$ and $q$ as before. This proves that the functional $\mathcal{J}_{\infty}$ reaches its minimum at $(\ui,\vi)$.

In order to close the circle we need to show that
\begin{equation}\label{ecco}
\max\left\{\frac{1}{R},\frac{1}{R^{r\Gamma+s(1-\Gamma)}}\right\}\le \mathcal{J}_{\infty}(u_{\infty}, v_{\infty}).
\end{equation}
For this end, let $x_0\in\Omega$ such that $u_{\infty}^{\Gamma}(x_0)\vi^{1-\Gamma}(x_0) = 1$  and set $a=u_{\infty}(x_0)$ and $b=\vi(x_0)$ so that
\[
\|\nabla \ui\|_{\elle{\infty}}\ge \frac{a}{d(x_0)}, \ \ \ |\ui|_{r,\infty}\ge \frac{a}{d(x_0)^r},
\]
and
\[
\|\nabla \vi\|_{\elle{\infty}}\ge \frac{b}{d(x_0)}, \ \ \ |\vi|_{s,\infty}\ge \frac{b}{d(x_0)^s}.
\]
where $d(x_0)=\mbox{dist}(x_0,\partial\Omega)$. Hence,
\begin{equation}\label{28-6}
\begin{array}{rcl}
  \mathcal{J}_{\infty}(u_{\infty}, v_{\infty}) & \ge & \mathrm{I}(a,b,d) \\
   & \defeq & \displaystyle \inf\left\{\max\left\{\frac{a}{d},\frac{a}{d^r},\frac{b}{d},\frac{b}{d^s}\right\}: \ a,b\in(0,\infty),\ a^{\Gamma}b^{1-\Gamma}=1, \ d\in(0,R]\right\}.
\end{array}
\end{equation}
To evaluate the infimum in the right hand side above let us consider at first $R\le1$ (recall that $R$ is the radius of the largest ball contained in $\Omega$). In this case $d\le\min\{d^r,d^s\}$ for $d\in(0, R]$. Thus,
$$
\begin{array}{rcl}
  \mathrm{I}(a,b,d) & = & \displaystyle \inf\left\{\frac{1}{d}\max\left\{a,b\right\}: \ a,b\in(0,\infty),\ a^{\Gamma}b^{1-\Gamma}=1, \ d\in(0, R]\right\} \\
   & = &\displaystyle \frac{1}{R}\displaystyle \inf\big\{\max\left\{a,b\right\}: \ a,b\in(0,\infty),\ a^{\Gamma}b^{1-\Gamma}=1\big\}\\
& = &\displaystyle \frac{1}{R} \ \ \ \mbox{if } R\le1,
\end{array}
$$
where the last equation comes from the fact that the \emph{inf} of the \emph{max} between $a$ and $b$ is reached for $a\equiv b$, that implies $1=a^{\Gamma}b^{1-\Gamma}=b^{\Gamma}b^{1-\Gamma}=b=a$.

On the other hand, if $R>1$ the infimum in \eqref{28-6} has to be achieved for values of $d$ bigger then $1$, hence
$$
\begin{array}{ccl}
  \mathrm{I}(a,b,d) & = & \displaystyle \inf\left\{\max\left\{\frac{a}{d^r},\frac{a^{-\frac{\Gamma}{1-\Gamma}}}{d^{s}}\right\}: \ a\in(0,\infty), \ d\in(0, R]\right\} \\
   & = & \displaystyle \inf\left\{\max\left\{\theta,\frac{\theta^{-\frac{\Gamma}{1-\Gamma}}}{d^{s+r\frac{\Gamma}{1-\Gamma}}}\right\}: \ \theta\in(0,\infty), \ d\in(0, R]\right\} \ \ \ \mbox{if} \ \ \ R>1.
\end{array}
$$
As before the infimum above is reached if the two argument of the maximum are equal, namely
\[
\theta=\frac{\theta^{-\frac{\Gamma}{1-\Gamma}}}{d^{s+r\frac{\Gamma}{1-\Gamma}}}\iff \theta=\frac1{d^{r\Gamma+s(1-\Gamma)}}.
\]
Thus, we got
\[
\mathrm{I}(a,b,d)=\inf\left\{\frac1{d^{r\Gamma+s(1-\Gamma)}} \ : \ d\in(0,R]\right\}=\frac{1}{R^{r\Gamma+s(1-\Gamma)}}  \ \ \ \mbox{if } R>1,
\]
and thus inequality \eqref{ecco} holds. Finally, putting together \eqref{28-6-bis} and \eqref{ecco} we conclude the proof of the Theorem.
\end{proof}

\section{Limiting PDE system in the viscosity sense}\label{Sec5}

Finally, we supply the proof of the Theorem \ref{MThm4}.

\begin{proof}[{\bf Proof of Theorem \ref{MThm4}}]

For the sake of brevity we prove only that the pair $(u_\infty,v_\infty)$ is a viscosity solution to the first equation of system \eqref{16-10bis}, since the other case is absolutely analogous. We begin by proving that the pair $(u_\infty,v_\infty)$ is a viscosity subsolution. Thus, by Definition \ref{DefVSlimeq} it is enough to prove that given $x_0\in\Omega$, for every $\phi$, $\psi \in C^2(\Omega)\cap C^1_0(\Omega)$ such that for $r>0$  sufficiently small
\begin{equation*}
\displaystyle
\left\{
\begin{array}{rclcl}
u_{\infty}(x)-\phi(x)<u_{\infty}(x_0)-\phi(x_0) & = &  0  \ \forall x \in B_r(x_0), \\
 v_{\infty}(x)-\psi(x)<v_{\infty}(x_0)-\psi(x_0) & = &  0  \ \forall x \in B_r(x_0),
 \end{array}
\right.
\end{equation*}
we have
$$
   \mathrm{G}^{r}_1[\phi(x_0), v_\infty(x_0)] \leq 0 \qquad \text{and} \qquad \mathrm{G}^r_2[\phi(x_0), v_\infty(x_0)]\leq 0.
$$

Now, consider $\xi \in C^\infty_0(\overline{\Omega})$ for which $\xi \equiv 1 $ in $B_r(x_0)$ and let $x_n \in B_{r_n}(x_0)$ be the maximum point of $u_n-\phi$ in $B_{r_n}(x_0)$, where $r_n=\frac{r}{n}$.
In this way, we define the test function
\begin{equation}
\label{1719}
\phi_n(x)=\phi(x)+\xi(x)\left(k_n+\frac{|x-x_n|^2}{n}\right),
\end{equation} where $k_n=u_n(x_n)-\phi(x_n)$.

Since $u_n \to u$ uniformly in $\Omega$ and $x_n \to x_0$, it is clear that $u_n-\phi_n$  attains a strict local maximum at $x_n$, $u_n(x_n)=\phi_n(x_n)$, and $\phi_n \to \phi$ uniformly in $\Omega$.

At this point, we are going to use Lemma \ref{EquivSols}. Indeed, since $(u_{n}, v_{n})$ is in particular a weak subsolution  of \eqref{Eq1}, by  Lemma \ref{EquivSols}, it is also a viscosity subsolution to the first equation of \eqref{Eq1}, so that
\begin{equation}\label{EqSubsolu_p}
  -\Delta_{p_n} \phi_n(x_{n}) +(-\Delta)^r_{p_n}\phi_n(x_{n})  \leq \lambda_{n}\phi_n^{\an-1}(x_{n})v_n^{\bn}(x_{n}) \quad \forall \,\,n \in \mathbb{N},
\end{equation}
and then
$$
   -|\Delta_{p_n} \phi_n(x_{n}) | + (\L_{p_n, r}^+)^{p_n-1}\phi_n(x_{n}) \leq    A_n + (\L_{p_n, r}^-)^{p_n-1}\phi_n(x_{n}),
$$
where $\mathcal{L}^\pm_{p_n, r}$ were defined in Lemma \ref{Lemlim-op}, and for the sake of simplicity, we denoted
\[
A_n=\frac{2\an}{\an+\bn}\lambda_{n}u_n^{\an-1}(x_{n})v_n^{\bn}(x_{n}),
\]
since $\phi_n(x_n)=u_n(x_n)$.

Moreover, by \eqref{1719}, remark that
\begin{equation}\label{1725}
\displaystyle
\left\{
\begin{array}{rclcl}
\nabla \phi_n(x_n)& = & \nabla\phi(x_n)+\nabla \xi(x_n) k_n & = &  \nabla\phi(x_n), \\
 D^2\phi_n(x_n)& = & D^2\phi(x_n)+D^2\xi(x_n) k_n + \frac{2}{n}I_N& = &  D^2\phi(x_n)+ \frac{2}{n}I_N,\\
 \Delta \phi_n(x_n)& = & \Delta\phi(x_n)+\Delta \xi(x_n) k_n+\frac{2N}{n}\xi(x_n)& = &  \Delta\phi(x_n)+ \frac{2N}{n},\\
 \Delta_\infty \phi_n(x_n)& = & \Delta_\infty\phi(x_n)+\frac{2}{n}|\nabla \phi (x_n)|^2.&  &
\end{array}
\right.
\end{equation}
since $\xi\equiv 1$ in $B_r(x_0)$ and $\{x_n\} \subset B_{r_n}(x_0).$

Thus, by combining \eqref{1725} and the latter inequality we get
$$
   \L_{p_n, r}^+\phi_n(x_{n}) \leq \left(A_n + (\L_{p_n, r}^-)^{p_n-1}\phi_n(x_{n}) + |\nabla \phi(x_{n})|^{p_n-4}|B_n|\right)^{\frac{1}{\pn-1}},
$$
where we considered
$$
   B_n = (p_n-2)\left( \frac{|\nabla \phi(x_{n})|^{2}\Delta \phi(x_{n})}{p_{n}-2}+ \Delta_{\infty} \phi(x_{n}) \right)+ |\nabla \phi(x_{n})|^{2}\frac{2(N+p_n-2)}{n},
$$

Hence, by the choice of $\phi_n$,  by Lemma \ref{Lemlim-op}, if we let $n\to\infty$, it is straightforward to see that
\begin{equation}
\label{1621}
  \L^{+}_{\infty, r}\phi(x_0) \leq \max\{\lambda_\infty\phi^{\Gamma}(x_0)v_\infty^{1-\Gamma}(x_0),- \L^{-}_{\infty, r}\phi(x_0), |\nabla \phi(x_0)|\},
\end{equation}
where we employed Theorem \ref{MThm2}, \eqref{pitt}, \eqref{l.infty.pm}, \eqref{1725}, and that $v_n(x_n)\to v_\infty(x_0)=\psi(x_0)$.

Now, it is clear that \eqref{1621} is equivalent to
$$
  \min\left\{\L^{+}_{\infty, r}\phi(x_0)-\lambda_\infty\phi^{\Gamma}(x_0)v_\infty^{1-\Gamma}(x_0), \L_{\infty, r}\phi(x_0), \L^{+}_{\infty, r}\phi(x_0)-|\nabla \phi(x_0)|\right\}\leq 0,
$$
so that by the definition of $\mathrm{G}^r_1$ we have
\begin{equation}\label{EqSubsolF_1}
\mathrm{G}^r_1[\phi(x_0), v_\infty(x_0)]\leq 0.
\end{equation}

Notice that it remains to prove that
\begin{equation}\label{EqSubsolF_2}
\mathrm{G}^r_2[\phi(x_0), v_\infty(x_0)]\leq 0.
\end{equation}

If $|\nabla \phi(x_0)| = 0$, since $u_\infty$, $v_\infty$ are non negative and all the other terms which appear in  $\mathrm{G}^r_2$, are null or non positive,   \eqref{EqSubsolF_2} is true.

Now, in the sequel, let us stress that in order to prove \eqref{EqSubsolF_2}, it is enough to show that at least one of its terms is non positive.
Thus, if $|\nabla \phi(x_0)| > 0$, one may assume that
{\small{
\begin{equation}\label{EqSubsolF_2assum}
  \min\left\{|\nabla \phi(x_0)| -\lambda_\infty\phi^{\Gamma}(x_0)v_\infty^{1-\Gamma}(x_0), |\nabla \phi(x_0)| - \max\left\{\mathcal{L}_{\infty, r}^{+} \phi(x_0), -\mathcal{L}_{\infty, r}^{-} \phi(x_0)\right\}\right\}>0,
\end{equation}}}
and thence, it is sufficient to prove that $-\Delta_{\infty} \phi(x_0) \leq 0$.

In this fashion, under these assumptions, by dividing inequality \eqref{EqSubsolu_p} by $(p_n-2)|\nabla \phi_n(x_n)|^{p_n-4}$ we arrive at
$$
\begin{array}{cc}
   -\frac{|\nabla \phi(x_{n})|^{2}}{p_n-2}\Delta\phi_n(x_{n}) - \Delta_{\infty}\phi_n(x_{n})+\frac{|\nabla \phi(x_{n})|^3}{p_n-2}\left[\left(\frac{\mathcal{L}_{p_n, r}^{+} \phi_n(x_{n})}{|\nabla \phi_n(x_{n})|}\right)^{p_n-1}-\left(\frac{\mathcal{L}_{p_n, r}^{-} \phi_n(x_{n})}{|\nabla \phi_n(x_{n})|}\right)^{p_n-1}\right]  & \\
    \leq\frac{1}{p_n-2}\left(\frac{ A_n^{\frac{1}{\pn-4}}}{|\nabla \phi(x_{n})|}\right)^{p_n-4},&
\end{array}
$$
since $|\nabla \phi_n(x_{n})|= |\nabla \phi(x_{n})|\neq 0$, for $n$ large enough.

Hence, by passing to the limit in the latter inequality and by combining \eqref{1725}, Lemma \ref{Lemlim-op} and \eqref{EqSubsolF_2assum} we finally obtain
$$
   -\Delta_{\infty} \phi(x_0) \leq 0,
$$
so that \eqref{EqSubsolF_2} holds true.

In this manner,  by combining \eqref{EqSubsolF_1} and \eqref{EqSubsolF_2} we get that $(u_\infty,v_\infty)$ is a viscosity subsolution to the first equation of \eqref{EqLim}, the other proof being analogous.

Now, in order to complete this proof, we want to show that, given $x_0 \in \Omega$ and test functions $\phi, \psi \in C^2(\Omega)\cap C^1_0(\Omega)$ such that for $r>0$ sufficiently small
\begin{equation*}
\displaystyle
\left\{
\begin{array}{rclcl}
u_{\infty}(x)-\phi(x)>u_{\infty}(x_0)-\phi(x_0) & = &  0  \ \forall x \in B_r(x_0), \\
 v_{\infty}(x)-\psi(x)>v_{\infty}(x_0)-\psi(x_0) & = &  0  \ \forall x \in B_r(x_0),
 \end{array}
\right.
\end{equation*}
 there holds that
$$
   \mathrm{G}^r_1[\phi(x_0), v_\infty(x_0)] \geq 0 \qquad \text{or} \qquad \mathrm{G}^r_2[\phi(x_0), v_\infty(x_0)]\geq 0,
$$
meaning that $(u_{\infty}, v_{\infty})$ is a viscosity supersolution of the first equation of \eqref{EqLim}. Once again, we stress that the argument for the second equation of the system is similar.

In an analogous manner to the case of subsolutions, we can find $\phi_n \in C^2(\Omega)\cap C^1_0(\Omega)$ and $\{x_{n}\}_{n \in \mathbb{N}} \subset B_{r_n}(x_0)$ such that:
\begin{enumerate}
  \item $x_{n} \to x_0$ as $n \to \infty$;
  \item $\phi_n \to \phi$ uniformly in $\Omega$ as $n \to \infty$;
  \item $u_{n}(x)-\phi(x)>u_{n}(x_0)-\phi(x_0)$ for all  $x \in B_{r_n}(x_0)$.
\end{enumerate}
Since $(u_{n}, v_{n})$ is a weak supersolution of \eqref{Eq1}, as a consequence of the Lemma \ref{EquivSols}, it is also a viscosity supersolution to the first equation of \eqref{Eq1}. Thus, by using the same notation as in the previous case,
\begin{equation}\label{EqSupersolu_p}
    -\Delta_{p_n} \phi_n(x_{n}) +(-\Delta)^r_{p_n}\phi_n(x_{n})  \geq A_n \quad \forall \,\,n \in \mathbb{N}.
\end{equation}
In this case,
$$
   -\Delta_{p_n} \phi_n(x_{n})  + (\L_{p_n, r}^+)^{p_n-1}\phi_n(x_{n}) \geq  A_n + (\L_{p_n, r}^-)^{p_n-1}\phi_n(x_{n}),
$$
implying that
$$
   |\nabla \phi_n(x_{n})|^{p_n-4}|B_n| + (\L_{p_n, r}^+)^{p_n-1}\phi_n(x_{n}) \geq A_n + (\L_{p_n, r}^-)^{p_n-1}\phi_n(x_{n}).
$$
Thus, it is clear that
$$
   \left(|\nabla \phi_n(x_{n})|^{p_n-4}|B_n| +  (\L_{p_n, r}^+)^{p_n-1}\phi_n(x_{n})\right)^{\frac{1}{p_n-1}} \geq    \left(A_n + (\L_{p_n, r}^-)^{p_n-1}\phi_n(x_{n})\right)^{\frac{1}{p_n-1}}.
$$
Now, by passing to the limit and by combining Lemma \ref{Lemlim-op}, Theorem \ref{MThm2}, sentences \eqref{pitt}, \eqref{l.infty.pm} and that $v_n(x_n)\to v_\infty(x_0)=\psi(x_0)$, one obtains that
$$
   \max\{|\nabla \phi(x_{0})|,  \L^{+}_{\infty, r}\phi(x_0)\} \geq \max\{\lambda_\infty\phi^{\Gamma}(x_0)v_\infty^{1-\Gamma}(x_0), -\L^{-}_{\infty, r}\phi(x_0)\}.
$$
At this point, we shall analyze separately the cases where
$$
   \L^{+}_{\infty, r}\phi(x_0) \geq |\nabla \phi(x_{0})| \quad \text{and} \quad  \L^{+}_{\infty, r}\phi(x_0) < |\nabla \phi(x_{0})|.
$$

First,  let us assume that $ \L^{+}_{\infty, r}\phi(x_0) \geq |\nabla \phi(x_{0})|$. In this case, it is clear that
  $$
     \L^{+}_{\infty, r}\phi(x_0) \geq \max\left\{\lambda_\infty\phi^{\Gamma}(x_0)v_\infty^{1-\Gamma}(x_0), -\L^{-}_{\infty, r}\phi(x_0)\right\},
  $$
  which is equivalent to
  $$
    \min\left\{\L^{+}_{\infty, r}\phi(x_0)-\lambda_\infty\phi^{\Gamma}(x_0)\psi^{1-\Gamma}(x_0), \L_{\infty, r}\phi(x_0)\right\}\geq 0.
$$
The previous inequality and the definition of  that $\mathrm{G}^r_1$, guarantee that
$$
   \mathrm{G}^r_1[\phi(x_0), \psi(x_0)] \geq 0.
$$

  On the other hand, if $ \L^{+}_{\infty, r}\phi(x_0) < |\nabla \phi(x_{0})|$, we must have $|\nabla \phi(x_{0})| \neq 0$, otherwise we should obtain
$$
   0\leq \L^{+}_{\infty, r}\phi(x_0)<0,
$$
clearly a contradiction. Thus, we obtain
$$
   |\nabla \phi(x_{0})| \geq \max\left\{\lambda_\infty\phi^{\Gamma}(x_0)v_\infty^{1-\Gamma}(x_0), -\L^{-}_{\infty, r}\phi(x_0)\right\},
$$
or equivalently,
\begin{equation}\label{EqlimSupersol}
  \min\left\{|\nabla \phi(x_{0})|-\lambda_\infty\phi^{\Gamma}(x_0)v_\infty^{1-\Gamma}(x_0), |\nabla \phi(x_{0})|+\L^-_{\infty, r}\phi(x_0)\right\}\geq 0.
\end{equation}

Now, let us stress that under the assumption $|\nabla \phi(x_0)| > 0$, it is clear that, for $n$ sufficiently large,  \eqref{EqSupersolu_p} can be rewritten as
$$
\begin{array}{cc}
   -\frac{|\nabla \phi(x_{n})|^{2}}{p_n-2}\Delta\phi_n(x_{n}) - \Delta_{\infty}\phi_n(x_{n})+\frac{|\nabla \phi(x_{n})|^3}{p_n-2}\left[\left(\frac{\mathcal{L}_{p_n, r}^{+} \phi_n(x_{n})}{|\nabla \phi(x_{n})|}\right)^{p_n-1}-\left(\frac{\mathcal{L}_{p_n, r}^{-} \phi_n(x_{n})}{|\nabla \phi(x_{n})|}\right)^{p_n-1}\right] &  \\
    \geq \frac{1}{p_n-2}\left(\frac{ A_n^{\frac{1}{\pn-4}}}{|\nabla \phi(x_{n})|}\right)^{p_n-4} ,&
\end{array}
$$
since $|\nabla \phi(x_{n})| \neq 0$ for $n$ large enough.

Therefore, after passing the limit and using Lemma \ref{Lemlim-op}, \eqref{1725}, the choice of $\phi_n$, and inequality \eqref{EqlimSupersol} we get
\begin{equation}\label{EqlimsupersolInfLap}
  -\Delta_{\infty} \phi(x_0) \geq 0.
\end{equation}

Finally, by combining \eqref{EqlimSupersol} and \eqref{EqlimsupersolInfLap} we conclude that
$$
   \mathrm{G}^r_2[\phi(x_0),v_\infty(x_0)]\geq 0.
$$

In both cases, $(u_{\infty},v_\infty)$ is a viscosity supersolution, thereby finishing the proof of the theorem.

\end{proof}

\section{Appendix}\label{Appendix}
In this final section we prove a very technical inequality, which we could not find in the specialized literature. Such a key tool has a pivotal role in the proof of the simplicity of the eigenvalue (Theorem \ref{MThsimp}), see Lemma \ref{bigin}.
\begin{lemma}
\label{lemma1ap} Consider  $f:[0,1]^3 \to \mathbb{R}$ given by
\[ f(x,y,z)=\left|(x^p+y^p)^{\frac{1}{p}}-(z^p+1)^{\frac{1}{p}}\right|-\left((1-x)^p+|y-z|^p\right)^{\frac{1}{p}},\]
where $p\geq 1$. Then $f(x,y,z)\leq 0$ for all $(x,y,z) \in [0,1]$.
\end{lemma}
\begin{proof}
Remark that it is enough to prove that $f(x,y,z)\leq 0$ on the boundary of $[0,1]^3$, in the interior points where $f$ fails to be of class $C^1$ or at its critical points.
For the sake of clarity we will split the proof in five steps:\\

{\bf Step 1:} $f(x,y,z)\leq 0$ if $y=z$ or $x^p+y^p=z^p+1$ and $(x,y,z)\in(0,1)^3$.

For the case $y=z$  consider $\phi_1(x,y)= f(x,y,y)=\phi_1(x,y)=(y^p+1)^{\frac{1}{p}}-(x^p+y^p)^{\frac{1}{p}}-1+x.$
We have that
\[\dfrac{\partial \phi_1}{\partial y}=y^{p-1}\left[\dfrac{1}{(y^p+1)^{\frac{p-1}{p}}}-\dfrac{1}{(x^p+y^p)^{\frac{p-1}{p}}}\right]< 0.\]
Since $\phi(x,0)= 0$ for $x,\in[0,1]$, the latter inequalities implies
\[\phi_1(x,y)< \phi(x,0)=0 \ \ \ \mbox{for} \ \ \ (x,y)\in [0,1]^2.\]

In the case $x^p+y^p=z^p+1$, $f(x,y,z)$ is trivially non-positive.\\

{\bf Step 2:} Consider $(x,y,z)\in (0,1)^3$ such that $y\neq z$ and $x^p+y^p \neq z^p+1$.

If $x^p+y^p>z^p+1$ it follows that
\[\dfrac{\partial f}{\partial x}= \dfrac{x^{p-1}}{\left(x^p+y^p\right)^{\frac{p-1}{p}}}+\dfrac{(1-x)^{p-1}}{\left((1-x)^p+|y-z|^p\right)^{\frac{p-1}{p}}}>0.\]
If $x^p+y^p<z^p+1$ and $y>z$ it follows that
\[\dfrac{\partial f}{\partial z}= \dfrac{z^{p-1}}{\left(z^p+1\right)^{\frac{p-1}{p}}}+\dfrac{(y-z)^{p-1}}{\left((1-x)^p+(y-z)^p\right)^{\frac{p-1}{p}}}>0.\]
So there are no critical points on these regions. Finally, if $x^p+y^p<z^p+1$ and $y<z$, it is easy to check that
$\frac{\partial f}{\partial z}=0$ if and only if $y=xz$, and $f(x,xz,y)=0$.\\

{\bf Step 3:} If $z=1$ or $z=0$ then $f(x,y,z)\leq 0$ for all $(x,y)\in [0,1]^2$.

First we consider the case $z=1$, and for this let us set $\phi_2(x,y)=f(x,y,1)=2^{\frac1p}-(x^p+y^p)^{\frac1p}-((1-x)^p+(1-y)^p)^{\frac{1}{p}}$. Studying the sing of $\frac{\partial\phi_2}{\partial x}$ (where it is defined) it follows that, for any $y\in[0,1]$ the function $x\to\phi_2(x,y)$ reaches its maximum at $x=y$ and $\phi_2(y,y)=0$.\\
In order to address the case $z=0$, let us set \[\phi_3(x,y)=f(x,y,0)=\left | (x^p+y^p)^{\frac{1}{p}}-1\right |-\left( (1-x)^p+y^p\right)^{\frac{1}{p}} \mbox{ for } (x,y)\in[0,1]^2.\]
Clearly, if $x^p+y^p=1$, then $\phi_3(x,y)<0$.

Moreover, if $x^p+y^p>1$ then it is cleat that
\[\dfrac{\partial \phi_3}{\partial x} = \dfrac{x^{p-1}}{\left (x^p+y^p\right)^{\frac{p-1}{p}}}+\dfrac{(1-x)^{p-1}}{\left ( (1-x)^p+y^p \right)^{\frac{p-1}{p}}}>0.\]
Thus we have that \[\phi_3(x,y)\leq \phi_3(1,y)\leq 0,\] since $(1+y^p)^{\frac{1}{p}}\leq (1+y)$. On the other hand if $x^p+y^p<1$ one has that
\[\dfrac{\partial \phi_3}{\partial y} = \dfrac{-y^{p-1}}{\left (x^p+y^p\right)^{\frac{p-1}{p}}}-y^p\dfrac{(1-x)^{p-1}}{\left ( (1-x)^p+y^p \right)^{\frac{p-1}{p}}}<0,\]
that in turn implies \[\phi_3(x,y)\leq \phi_3(x,0)=0,\]
completing the proof of Step 3.\\

{\bf Step 4:} If $y=1$ or $y=0$ then $f(x,y,z)\leq 0$ for all $(x,z)\in [0,1]^2$.

Indeed, for $y=1$, consider $\phi_4(x,z)=f(x,1,z)$, i.e.,
\[\phi_4(x,z)=\left | (x^p+1)^{\frac{1}{p}}-(z^p+1)^{\frac{1}{p}}\right |- \left ((1-x)^p+(1-z)^p\right)^{\frac{1}{p}}.\]
Thanks to the symmetry, it is not restrictive assume that $x>z$. We have that,
\[\dfrac{\partial \phi_4}{\partial x}=\dfrac{x^{p-1}}{(x^p+1)^{\frac{p-1}{p}}}+\dfrac{(1-x)^{p-1}}{\left ((1-x)^p+(1-z)^p\right)^{\frac{p-1}{p}}}>0\]
for all $(x,z)\in (0,1)^2$ such that $x> z$. Thus we deduce that
\[\phi_4(x,z)\leq \phi_4(1,z)\leq \phi_4(1,1)=0,\]
where the last inequality follows from
\[\dfrac{\partial \phi_4}{\partial z}(1,z)=\dfrac{-z^{p-1}}{(z^p+1)^{\frac{p-1}{p}}}+1>0, \ \forall z\in [0,1].\]

Now, we analyse the case $y=0$ and for this let us define

\[\phi_5(x,z)=(z^p+1)^{\frac{1}{p}}-x- \left ((1-x)^p+z^p\right)^{\frac{1}{p}},\]
 where $(x,z)\in [0,1]^2$ and which clearly satisfies $\phi_5(x,z)=f(x,0,z)$.
Observe that one has
\[\dfrac{\partial \phi_5}{\partial z}=\dfrac{z^{p-1}}{\left(1+z^p\right)^{\frac{p-1}{p}}}-\dfrac{z^{p-1}}{\left((1-x)^p+z^p\right)^{\frac{p-1}{p}}}<0\]
and consequently, by the continuity of $\phi_5$ we get
\[\phi_5(x,z)\leq \phi_5(x,0)=0 \ \ \ \mbox{in} \ \ \ [0,1]^2, \]
what finishes the proof of Step 4.\\

{\bf Step 5:} If $x=1$ or $x=0$ then $f(x,y,z)\leq 0$ for all $(y,z)\in [0,1]^2$.

For $x=1$ we define
\[\phi_6(y,z)=f(1,y,z)=\left | (1+y^p)^{\frac{1}{p}}-(z^p+1)^{\frac{1}{p}}\right | -|y-z|, \mbox{ for } (y,z)\in[0,1]^2.\]
Since $\phi_6(y,z)=\phi_6(z,y)$ if we prove that $\phi_6(y,z)\leq 0$ for $y\geq z$. With this considerations in mind, observe that
\[ \dfrac{\partial \phi_6}{\partial y}= \dfrac{y^{p-1}}{\left ( 1+y^p\right )^{\frac{p-1}{p}}}-1 <0,\]
so that $\phi_6(y,z)\leq \phi_6(z,z)=0$ as we desired.\\
Finally, for $x=0$, we define
\[\phi_7(y,z)=f(0,y,z)=(z^p+1)^{\frac{1}{p}}-y -(1+|y-z|^p)^{\frac{1}{p}}, \mbox{ where } (y,z)\in [0,1].\]
Hence, if $y\geq z$, since $(y^p+1)^{\frac{1}{p}}\leq y+1$, we arrive at
\begin{align*}
\phi_7(y,z)&\leq  (y^p+1)^{\frac{1}{p}}-y -(1+(y-z)^p)^{\frac{1}{p}} \\
&\leq  1 -(1+(y-z)^p)^{\frac{1}{p}} \leq 0.
\end{align*}
Now, if $y<z$, remark that
\[ \dfrac{\partial \phi_7}{\partial y}=-1+\dfrac{(z-y)^{p-1}}{\left (1+(z-y)^p\right)^{\frac{p-1}{p}}} <0.\]
Hence, by the continuity of $\phi_7$, given $(y,z)\in[0,1]^2$ for which $y\leq z$, there holds that $\phi_7(y,z)\leq \phi_7(0,z)=0$.
In this way, we also have that $f(0,y,z)\leq 0 \ \forall (y,z)\in [0,1]^2$, completing the proof of Step 5.\\

Therefore, we have proved that $f(x,y,z)\le 0$ both on $\partial[0,1]^3$ and in $[0,1]^3$.
\end{proof}

In conclusion, for the convenience of the reader, we provide a proof to a result which plays a key role in the text, namely Lemma \ref{Lemlim-op}.

\begin{lemma}
\label{1501}
Let $\Omega \subset \mathbb{R}^N$ be bounded, $\{f_n\}\subset L^\infty(\mathbb{R}^N)$ and $\{p_n\}\subset (1,\infty)$.
Suppose that:
\begin{itemize}
\item[(a)] $p_n\leq p_{n+1} \ \forall n \in \mathbb{N}$ and $p_n \to \infty$ if $n\to \infty$;
\item[(b)] There exists $f\in L^\infty(\Omega)$ for which $f_n \to f$ in $L^\infty(\Omega)$ if $n\to \infty$;
\item[(c)] There exists $C\geq 0$ such that
\[\lim_{n\to \infty} \| f_n\|_{L^{p_n}(\mathbb{R}^N\setminus \Omega)}=C.\]
\end{itemize}
Then,
\[\lim_{n\to \infty} \| f_n\|_{L^{p_n}(\mathbb{R}^N)}=\max \{\|f\|_{L^\infty(\Omega)},C\}.\]
\end{lemma}
\begin{proof}
As a first step, recall that, since $\Omega$ is bounded,
\[\lim_{n\to \infty} \|f\|_{L^{p_n}(\Omega)}=\|f\|_{L^{\infty}(\Omega)}.\]

Moreover, observe that
\[ \left|\|f\|_{L^{\infty}(\Omega)}-\|f_n\|_{L^{p_n}(\Omega)}\right|\leq \left|\|f\|_{L^{\infty}(\Omega)}-\|f\|_{L^{p_n}(\Omega)}\right|+\|f-f_n\|_{L^{p_n}(\Omega)}.\]

Thus, given $\epsilon>0$, set $n_0\in \mathbb{N}$ for which if $n\geq n_0$ then
\[\left|\|f\|_{L^{\infty}(\Omega)}-\|f\|_{L^{p_n}(\Omega)}\right|<\frac{\epsilon}{2}, |\Omega|^{\frac{1}{p_n}}\leq 2 \mbox{ and } |f_n-f|<\dfrac{\epsilon}{4} \mbox{ a.e. in } \Omega.\]
By combining the latter inequalities we arrive at
\begin{equation}
\label{1057}
\displaystyle \lim_{n\to \infty} \|f_n\|_{L^{p_n}(\Omega)}=\|f\|_{L^\infty(\Omega)}.
\end{equation}
Finally, recall that given to sequences of non negative real numbers, $\{a_n\}\subset\mathbb{R}$ and ${b_n}$, where $a_n \to a$ and $b_n\to b$ we have
\[\lim_{n\to\infty}\left(a_n^{p_n}+b_n^{p_n}\right)^{\frac{1}{p_n}}=\max\{a,b\}.\]
Therefore the result follows by combining \eqref{1057} and the convergence given in $(c)$.
\end{proof}


\noindent{\bf Acknowledgments.}  The authors would like to thank Marcelo Fernandes Furtado for several insightful comments and suggestions throughout the elaboration of this manuscript. Stefano Buccheri 
Stefano Buccheri has been partially supported by Coordena\c{c}\~{a}o de Aperfei\c{c}oamento de Pessoal de N\'{i}vel Superior (PNPD/CAPES-UnB-Brazil), Grant 88887.363582/2019-00, and by the Austrian Science Fund (FWF) project F65.
 Jo\~{a}o Vitor da Silva have been partially supported by Coordena\c{c}\~{a}o de Aperfei\c{c}oamento de Pessoal de N\'{i}vel Superior (PNPD/CAPES-UnB-Brazil), Grant 88887.357992/2019-00 and by Conselho Nacional de Desenvolvimento Cient\'{i}fico e Tecnol\'{o}gico (CNPq-Brazil) under Grant No. 310303/2019-2. Lu\'{i}s Henrique de Miranda was partially supported by CNPq-Brazil, FAPDF-Brazil and FEMAT/DF Brazil, Grants 407952/2016-0, 22968.93.32974.22052018 and 01/2018.

{We would like to thank the anonymous Referee for insightful comments and suggestions which improved the final outcome of this manuscript}.

\end{document}